\documentclass[smallextended,francais,envcountsect,envcountsame]{svjour3}       %
\smartqed  %
\usepackage{graphicx}
 \usepackage{mathptmx}      %
\usepackage[francais,english]{babel} %
\usepackage[OT1]{fontenc}       %
\usepackage[utf8]{inputenc}
\usepackage{epsfig}             %
\usepackage{amssymb} %
\usepackage{bm} %

\input{comp.def}
 \journalname{}

\begin{document}

\title{Compactification d'espaces de repr{\'e}sentations\\ 
de groupes de type fini.%
\thanks{avec le soutien de l'ANR Repsurf :  ANR-06-BLAN-0311}
}

\titlerunning{Compactification d'espaces de repr{\'e}sentations}        

\author{Anne Parreau}

\institute{Institut Fourier, Universit{\'e} Grenoble I et CNRS,
BP 74,
38402 Saint-Martin-d'H{\`e}res cedex, France.\\
\email{Anne.Parreau@ujf-grenoble.fr}
}

\date{}

\maketitle

\selectlanguage{english}

\begin{abstract}
Let $\Ga$ be a finitely generated group and $G$ be a noncompact
semisimple connected real Lie group with finite center. 
We consider the space $\XsepGaG$ of conjugacy classes of
semisimple representations of $\Ga$ into $G$, which is the maximal
Hausdorff quotient of $\Hom(\Ga,G)/G$.
 We define the {\it translation vector} of $g\in G$, with value in a
Weyl chamber, as a natural refinement of the translation
 length of $g$ in the symmetric space associated with
 $G$.
We construct a compactification of $\XsepGaG$, induced by the marked
translation vector spectrum, generalizing Thurston's compactification
of the Teichm\"uller space.
We show that the boundary points are projectivized marked translation
vector spectra of actions of $\Ga$ on affine buildings with no global
fixed point.
An analoguous result holds for  any reductive group $G$ over a local
field. 

\end{abstract}

\selectlanguage{francais}

\section{Introduction}

\paragraph{Espaces de repr{\'e}sentations.}
Soit $\Ga$ un groupe infini de type fini.
Soit $G$ un groupe de Lie semisimple r{\'e}el, connexe, non compact, de
centre fini.
On s'int{\'e}resse {\`a} l'espace $\RGaG=\Hom(\Ga,G)$ des repr{\'e}sentations
de $\Ga$ dans $G$,  muni de la topologie de la convergence simple
et de l'action de $G$ par conjugaison.
 L'espace topologique quotient $\RGaGsG$ n'{\'e}tant pas s{\'e}par{\'e},
 on le remplace par son plus gros quotient s{\'e}par{\'e}
 $\XsepGaG=\RGaGssG$, que l'on peut d{\'e}crire  de la
 mani{\`e}re suivante (voir section
 \ref{ss- espace de rep et quotient}, et \cite{ParRepcr}).
On consid{\`e}re l'action de $G$ par isom{\'e}tries sur son espace sym{\'e}trique
(sans facteur compact) $\Es$ associ{\'e}. On note $\bordinf\Es$ le bord {\`a}
l'infini de $\Es$.
Une repr{\'e}sentation $\rho:\Ga\fleche G$ est dite
{\em compl{\`e}tement r{\'e}ductible} (cr) si
pour tout  $\alpha\in\bordinf\Es$ fix{\'e} par $\rho$,  il existe
$\beta\in\bordinf \Es$ fix{\'e} par $\rho$ et oppos{\'e} {\`a} $\alpha$ (ie. joint
{\`a} $\alpha$ par une g{\'e}od{\'e}sique dans $\Es$).
Cette propri{\'e}t{\'e} g{\'e}om{\'e}trique (introduite par J.P~Serre \cite{SerreCR})
est {\'e}quivalente aux notions alg{\'e}briques classiques de r{\'e}ductivit{\'e} ou
semisimplicit{\'e}.

L'espace $\XsepGaG$ s'identifie alors naturellement au sous-espace $\RcrGaGsG$ de
 $\RGaGsG$ form{\'e} par les classes de repr{\'e}sentations cr. 
Cet espace est consid{\'e}r{\'e} de mani{\`e}re classique 
dans la litt{\'e}rature. Il est l{\'e}g{\`e}rement ``plus gros'' que
le quotient alg{\'e}bro-g{\'e}om{\'e}trique usuel
(voir section \ref{ss- espace de rep et quotient}).
On s'int{\'e}resse aussi  plus sp{\'e}cialement
 au sous-espace 
$\XsepfdGaG$ de $\XsepGaG$ form{\'e} des classes de
repr{\'e}sentations fid{\`e}les et discr{\`e}tes.

\paragraph{Quelques exemples.}
Si $\Ga$ est le groupe fondamental d'une surface $\Sf$ connexe
orien\-t{\'e}e ferm{\'e}e, de genre $g\geq 2$, et $G=\PSL_2(\RR)$, alors $\Es$
est le plan hyperbolique $\HH^2$, et l'espace de Teichm{\"u}ller $\Teich$
des structures hyperboliques (marqu{\'e}es) sur $\Sf$
 s'identifie naturellement {\`a} l'une des
composantes connexes de $\XsepGaG$, incluse dans $\XsepfdGaG$.
Si $\Ga=\pi_1(\Sf)$ et $G=\PSLn(\RR)$, alors d'apr{\`e}s \cite{Hitchin}
et \cite{Labourie} la composante de Hitchin
de l'espace des modules de $G$-fibr{\'e}s plats sur $\Sf$ 
est 
une composante connexe de $\XsepGaG$, incluse dans $\XsepfdGaG$,
 et c'est
une cellule de dimension $(2g-2)\dim G$.
Pour $n=3$ Goldman et Cho{\"\i} ont montr{\'e} dans \cite{ChGo} qu'elle
s'identifie {\`a}  l'espace des structures projectives r{\'e}elles
convexes (marqu{\'e}es) sur la surface $\Sf$.
Par contre, si $\Ga$ est un r{\'e}seau irr{\'e}ductible d'un groupe de Lie $H$
semisimple r{\'e}el, connexe, sans facteurs compacts, de centre fini, avec
$\rang(H)\geq 2$,
le th{\'e}or{\`e}me de super-rigidit{\'e} de
Margulis \cite{Margulis}  entra{\^\i}ne que  l'espace
$\Xsep(\Ga,\SLn(\CC))$ est fini pour tout $n$
(voir aussi \cite{Bass}).
Les r{\'e}seaux des groupes de Lie r{\'e}els semisimples (connexes, sans
facteurs compacts, de centre fini) $H$ de rang $1$ non localement
isomorphes {\`a} $\PSL_2(\RR)$ (comme $\SO_0(n,1)$ pour $n\geq 3$) n'ont
pas de d{\'e}formations non triviales en des r{\'e}seaux dans $H$ par les
r{\'e}sultats de rigidit{\'e} de Mostow \cite{Mostow}. Par contre ils ont dans
certains cas de riches espaces de d{\'e}formations fid{\`e}les et discr{\`e}tes
dans d'autres groupes $G$, correspondant par exemple aux structures
conformes ($G=\SO_0(n+1,1)$) et projectives r{\'e}elles
($G=\PGL_{n+1}(\RR)$) marqu{\'e}es sur une vari{\'e}t{\'e} $M$ de groupe
fondamental $\Ga$ (voir par exemple \cite{JoMi}).

\paragraph{Objectif.}
W.~Thurston a construit une compactification de l'espace de Teichm{\"u}ller
(voir par exemple \cite{FLP}).
Plus g{\'e}n{\'e}ralement, Morgan et Shalen (voir \cite{MoSh}, \cite{Morgan},
\cite{Chiswell}) puis Bestvina \cite{Bestvina} et Paulin
\cite{PauInv}, ont g{\'e}n{\'e}ralis{\'e} la compactification de Thurston pour
 $\XsepfdGaG$ avec $G=\SO_0(n,1)$.
Les points du bord proviennent d'actions de $\Ga$ sur des {\em arbres
r{\'e}els}, {\`a} ``petits'' (i.e. virtuellement cycliques) stabilisateurs
d'ar{\^e}tes. L'un des outils fondamentaux est le {\em spectre marqu{\'e} des
 longueurs}.
Le but de ce travail est de g{\'e}n{\'e}raliser cette compactification pour
$G$ de rang sup{\'e}rieur ou {\'e}gal {\`a} $2$.
La possibilit{\'e} d'une telle compactification a {\'e}t{\'e} envisag{\'e}e auparavant
par de nombreuses personnes (Gromov, Paulin \cite{PauHab},
Kleiner-Leeb, \ldots).
En rang sup{\'e}rieur ou {\'e}gal {\`a} $2$, il convient de remplacer les arbres
r{\'e}els par une cat{\'e}gorie plus g{\'e}n{\'e}rale d'espaces m{\'e}triques, les {\em
immeubles affines}, ou encore {\em euclidiens}.
Expliquons maintenant plus en d{\'e}tail le r{\'e}sultat pr{\'e}sent{\'e} dans cet
article. 

Pour compactifier $\XsepGaG$, on introduit (en section \ref{s- VdT}) un
raffinement de la notion de longueur de
translation, le {\em  vecteur de translation}.
On se fixe une chambre de Weyl ferm{\'e}e $\Cb$ de $\Es$. 
On consid{\`e}re la projection naturelle $\Cd:\Es\times\Es\fleche\Cb$,
qui raffine la distance usuelle.
Le {\em vecteur de translation} $v(g)\in\Cb$ de $g\in G$ est
l'unique vecteur de longueur minimale dans l'adh{\'e}rence de 
${\{\Cd(x,gx),\ x\in \Es\}}$.
Alg{\'e}briquement, c'est la ``projection de Jordan'' de $g$ \cite{Benoist}.
Pour $G=\SLn(\RR)$, cet invariant de conjugaison est {\'e}gal {\`a} la suite
d{\'e}croissante des logarithmes des modules des valeurs propres de $g$. 
On consid{\`e}re l'application suivante 
$$\V :
\begin{array}{ccl}
\XsepGaG & \fleche & \CGa\\[0ex]
[\rho] & \mapsto &v\circ\rho
\end{array} \; .$$
Soit $\PCGa$ le projectifi{\'e} de $\CGa$, c'est-{\`a}-dire l'espace topologique
 quotient de $\CGao$ par multiplication dans $\RRpet$, et  $\PP: \CGao
 \fleche \PCGa$ la projection canonique.
On d{\'e}montre alors le th{\'e}or{\`e}me suivant (voir th{\'e}or{\`e}me \ref{theo- comp}
pour un {\'e}nonc{\'e} plus pr{\'e}cis).

\begin{theointro} \label{theointro- interpretation points du bord}
L'application continue $$\PP\circ\V :\XsepGaG - \V^{-1}(0) \fleche
\PCGa$$ induit une compactification $\XseptGaG$ de $\XsepGaG$, dont le
bord $\bordinf\XsepGaG\subset \PCGa$ est form{\'e} de points de la forme
$x=[\vro]$, o{\`u} $\rho$ est une action de $\Ga$ sur un immeuble affine
(pas n{\'e}cessairement discret), sans point fixe global.

L'action du groupe $\Out(\Ga)$ des automorphismes ext{\'e}rieurs
de $\Ga$ sur $\XsepGaG$ s'{\'e}tend naturellement {\`a} $\XseptGaG$.
\end{theointro}

Si  $\Ga$ n'a pas de sous-groupe d'indice fini
poss{\'e}dant un sous-groupe ab{\'e}lien distingu{\'e} infini 
(par exemple si $\Ga$ est un sous-groupe discret fortement
irr{\'e}\-duc\-tible de $\SLn(\RR)$),
alors le sous-espace $\XsepfdGaG$ des classes de repr{\'e}sentations fid{\`e}les et
discr{\`e}tes est un ferm{\'e} de $\XsepGaG$ (section
\ref{ss- comp Xfd}).
La compactification  de $\XsepGaG$ cons\-trui\-te induit donc une
compactification naturelle $\XsepfdtGaG$ de $\XsepfdGaG$.
Dans ce cas, F.~Paulin a d{\'e}montr{\'e} que les actions $\rho$ dont
proviennent les points du bord sont {\`a} stabilisateurs de germes
d'appartements virtuellement r{\'e}solubles \cite{PauDg}.

\medskip

En particulier, lorsque $\Ga$ est un groupe de surface et
$G=\PSLn(\RR)$, ceci donne une compactification de la composante de
Hitchin.
On a {\'e}tudi{\'e} dans \cite{ParTh} des exemples de d{\'e}g{\'e}n{\'e}rescences de
structures projectives r{\'e}elles (cas o{\`u} $n=3$) par
pliage, et calcul{\'e} les spectres limites correspondants.

\paragraph{Sur la structure de la preuve.}
Les actions sur des immeubles affines sont obtenues par passage {\`a} des
c{\^o}nes asymptotiques, comme dans \cite{PauDg} et \cite{KaLe}.
L'essentiel de la difficult{\'e} consiste ici, premi{\`e}rement, {\`a} d{\'e}montrer la
continuit{\'e} du vecteur de translation par passage {\`a} un c{\^o}ne
asymptotique (Prop. \ref{prop- passage au cone}), et deuxi{\`e}mement, que
le spectre limite obtenu est non nul, fait crucial pour cette
compactification.

Cette derni{\`e}re propri{\'e}t{\'e} est obtenue en utilisant un mod{\`e}le alg{\'e}brique
pour le c{\^o}ne asymptotique.  On peut en effet voir 
l'action limite comme
provenant d'une repr{\'e}sentation $\rho:\Ga\fleche\SLn(\KKom)$, 
agissant sur son immeuble de Bruhat-Tits, o{\`u}
$\KKom$ est un corps valu{\'e},  ni localement compact, ni {\`a} valuation
discr{\`e}te, obtenu {\`a} partir de $\RR$ par un proc{\'e}d{\'e} de passage {\`a} un
c{\^o}ne asymptotique (en utilisant un ultrafiltre $\om$).
Cette construction est d{\'e}taill{\'e}e dans la section \ref{s- modele
alg de Esom}.
La non nullit{\'e} du spectre d{\'e}coule alors de l'absence de point fixe
global par \cite{ParEll}.

\bigskip

On obtient {\'e}galement les r{\'e}sultats analogues
dans le cadre des groupes r{\'e}ductifs sur les corps locaux non
archim{\'e}diens. 

\section{Pr{\'e}liminaires}

Cette section est consacr{\'e}e {\`a} des pr{\'e}liminaires g{\'e}om{\'e}triques. On
introduit les notations utilis{\'e}es et on rappelle ou on {\'e}tablit les
propri{\'e}t{\'e}s dont on aura besoin, 
d'abord pour les espaces m{\'e}triques $\CAT(0)$ (section \ref{ss-
  CAT(0)}), puis plus sp{\'e}cifiquement pour les espaces sym{\'e}triques et
les immeubles affines (section \ref{ss- prel ES IMM}). 
Enfin on rappelle la notion de c{\^o}ne asymptotique d'un espace m{\'e}trique
et les propri{\'e}t{\'e}s des c{\^o}nes asymptotiques des espaces sym{\'e}triques et
des immeubles qu'on utilisera (section \ref{ss- prel conas}).

\subsection{Espaces m{\'e}triques  $\CAT(0)$}
\label{ss- CAT(0)}

Dans cette section $\Es$ est un espace m{\'e}trique $\CAT(0)$.
On renvoie  {\`a} \cite{BrHa} comme r{\'e}f{\'e}rence g{\'e}n{\'e}rale.
Etant donn{\'e}s $x,y\in\Es$, on notera $[x,y]$ ou bien $xy$
l'unique segment g{\'e}od{\'e}sique joignant $x$ {\`a} $y$.
Etant donn{\'e}s  $x, y, z\in \Es$ avec $y,z\neq x$, on notera
$\angletilde_x(y,z)$ l'angle de comparaison en $x$  entre $y$ et $z$, 
et $\angle_x(y,z)$ l'angle en $x$ entre $y$ et $z$.
L'espace $\Es$ se d{\'e}compose canoniquement en produit $\Es_0\times
\Es'$, o{\`u} $\Es_0$ est  euclidien et $\Es'$ est un espace
m{\'e}trique $\CAT(0)$ sans facteur euclidien.
Les isom{\'e}tries de $\Es$ pr{\'e}servent cette d{\'e}composition \cite[Thm
  II.6.15]{BrHa}.
On appellera $\Es_0$ le {\em facteur euclidien (maximal)} de $\Es$.

Le groupe $\Isom(\Es)$ des isom{\'e}tries de $\Es$ est muni de la
topologie de la convergence uniforme sur les born{\'e}s (qui est
m{\'e}trisable).
Lorsque $\Es$ est propre (i.e.~si ses boules ferm{\'e}es sont compactes)
 $\Isom(\Es)$ est localement compact, et agit 
proprement sur $\Es$ (c'est-{\`a}-dire que si $x\in\Es$ et
$M\geq 0$, alors $\{g\in\Isom(\Es),\ d(x,gx)\leq M\}$ est un compact).

\paragraph{Faisceaux.}
Soit $r:\RR\fleche\Es$ une g{\'e}od{\'e}sique.
Le {\em faisceau} $\Fr$ 
est la r{\'e}union des g{\'e}o\-d{\'e}\-si\-ques parall{\`e}les {\`a} $r$.  
C'est un sous-espace  convexe et il se
d{\'e}compose  \cite[Thm II.2.14]{BrHa} canoniquement en un produit
$\Cr\times r$, o{\`u} $\Cr\subset\Es$ est un convexe et $\{x\}\times r$ est
une g{\'e}od{\'e}sique parall{\`e}le {\`a} $r$ pour tout $x\in\Cr$.
Si $\Es$ est complet, alors $\Fr$ et $\Cr$ aussi.

\subsubsection{D{\'e}placement et longueur de translation \cite[II.6]{BrHa}.}
Soit $g\in\Isom(\Es)$.
On appelle fonction {\em de d{\'e}placement} de $g$ la fonction convexe
$$\fonction{d_g}{\Es}{\RRp}{x}{d(x,gx)}$$  
et  {\em longueur de translation} de
$g$ le r{\'e}el positif
$$\ell(g)=\inf_{x\in \Es}d_g(x).$$
L'{\em ensemble minimal} de $g$  est  le convexe ferm{\'e}
({\'e}ven\-tu\-el\-lement vide) 
$$\Min(g)=\{x\in\Es,\ d(x,gx)=\ell(g)\}\;.$$ 
On dit que $g$ est {\em semi-simple} si $\Min(g)\neq\emptyset$, 
{\em parabolique} sinon.  
Dans le premier cas, ou bien si $\ell(g)=0$ et $g$ poss{\`e}de un point
fixe ($g$ est alors dite {\em
elliptique})  ou bien $\ell(g)>0$, et
translate alors une g{\'e}od{\'e}sique  ($g$ est alors dite {\em
axiale}). 
Si $g$ envoie une g{\'e}od{\'e}sique $r$ sur une g{\'e}od{\'e}sique parall{\`e}le,
alors le faisceau $\Fr=\Cr\times r$ est stable par $g$ et $g$ agit
dessus  comme $(\gp,t)$ o{\`u} $\gp\in\Isom(\Cr)$
et $t\in\RR$ d{\'e}signe la translation $r(s)\mapsto r(s+t)$ de $r$.
On a alors $\Min(g)=\Min(\gp)\times r \subset \Fr$ et
\begin{equation}
\label{eq- faisceau stable et ell(g)}   
\ell(g)=\ell(g_{|\Fr})=\sqrt{\ell(\gp)^2+t^2} \; .
\end{equation}

On montre facilement que pour tout $x\in \Es$ la limite
$\lim_{k\tend\pinfty}\frac{1}{k} d(x, g^kx)$ existe et ne d{\'e}pend pas
de $x\in\Es$ (voir \cite[II.6.6]{BrHa}). Comme $\frac{1}{k} d(x,
g^kx)\leq d_g(x)$ pour tout $k$, on a de plus 
\begin{equation}
\label{eq- d(x, g^k x)}
\lim_{k\tend\pinfty}\frac{1}{k} d(x, g^kx)\leq \ell(g)
\end{equation}
avec {\'e}galit{\'e}  si $g$ semisimple.
On obtient ais{\'e}ment les propri{\'e}t{\'e}s suivantes.

\begin{lemm} 
\label{lemm- Min(gk) reste a distance bornee}
Soit $\gk\tend  g$ dans $\Isom(\Es)$. 

\begin{enumerate}
\item  \label{item- dgk tend vers dg}
On a $d_{\gk}\tend d_g$ uniform{\'e}ment sur les born{\'e}s de $\Es$.

 \item \label{item- limsup lgk leq lg} (semicontinuit{\'e})
On a  $\limsup_k\ell(\gk) \leq \ell(g)$.

 \item \label{item- lim Min(gk) est inclus dans Min(g)}
Soit $M$ l'ensemble des valeurs d'adh{\'e}rence des suites $(\xk)_k$
dans $\Es$ telles que $\xk\in\Min(\gk)$ pour tout $k\in\NN$. 
Alors $M\subset \Min(g)$.
\end{enumerate}
 
En particulier, si $\Es$ est propre et si $\Min(\gk)$ est non vide 
et reste {\`a} distance born{\'e}e de $\xo\in\Es$ fix{\'e}, alors quitte {\`a}
extraire $\Min(\gk)$ converge 
(pour la topologie de Hausdorff point{\'e}e \cite{GromovMetric}) 
vers un sous-ensemble non vide
$N$ de $\Min(g)$.
En particulier $g$ est semisimple et $\ell(g)=\lim_k \ell({\gk})$.
\qed \end{lemm}

\subsubsection{Bord a l'infini.}
On suppose d{\'e}sormais $\Es$ complet.
On notera $\bordinf \Es$ le bord ({\`a} l'infini) de $\Es$ et $\ov{\Es}=\Es\cup\bordinf \Es$.
Si $r:\RRp\fleche\Es$ est un rayon g{\'e}od{\'e}sique, on note $r(\infty)$ le
point de $\bordinf\Es$ d{\'e}fini par $r$, qu'on appellera l'{\em
  extr{\'e}mit{\'e}} de $r$.

\begin{defi}[Ombre $\OmbxyD x y D$ d'une boule]
\label{def- ombre}
Soit $x\in\ov \Es$.  Pour tout $y\in\Es$, et pour tout $D\geq 0$, on
notera $\OmbxyD x y D$ l'{\em ombre vue du point $x$ de la boule
  $B(y,D)$}, c'est-{\`a}-dire l'ensemble des $r(\pinfty)$, o{\`u} $r$ est une
g{\'e}od{\'e}sique issue de $x$ telle que $d(y,r)\leq D$.
\end{defi}

L'espace $\ov{\Es}$ est muni de la topologie des
c{\^o}nes \cite[II.8.6]{BrHa}. 
L'action de $\Isom(\Es)$ s'{\'e}tend contin{\^u}ment sur $\ov{\Es}$.
Soient $\xi,\xi'\in\ov{\Es}$.
Pour $x\in\Es$, on notera $\angle_x(\xi,\xi')$ l'angle en $x$ entre
$\xi$ et $\xi'$.
On notera $\angleTits(\xi,\xi')=\sup_{x\in \Es}\angle_x(\xi,\xi')$
l'{\em angle de Tits}, qui est une m{\'e}trique sur $\bordinf\Es$ 
(voir II 9.4  et 9.5 dans \cite{BrHa}).
Si $\xipp$ et $\xim$ sont les deux extr{\'e}mit{\'e}s d'une g{\'e}od{\'e}sique $r$,
on dira que $\xipp$ et $\xim$ sont {\em  oppos{\'e}s} 
et on notera $\Fximp=\Fr$ (alors $\angleTits(\xim,\xipp)=\pi$).

\subsubsection{Quelques propri{\'e}t{\'e}s des isom{\'e}tries relatives au bord {\`a} l'infini.}

\begin{prop}[Points fixes {\`a} l'infini d'une isom{\'e}trie axiale]
\label{prop- points fixes isom axiale} 
Soit $g\in\Isom(\Es)$  translatant (non trivialement) une g{\'e}od{\'e}sique
$r$.  Notons $\xim=r(\minfty)$ et $\xipp=r(\pinfty)$. 
Si $g$ fixe  $\xi\in\bordinf\Es$, alors
$\angleTits(\xim,\xi)+\angleTits(\xi,\xipp)=\pi$. 
En particulier, 
si $\xi$ est oppos{\'e} {\`a} $\xim$, alors $\xi=\xipp$.
\end{prop}
\begin{preuve} 
On a $\angleTits(\xim,\xi)=\lim \angle_{r(t)}(\xim,\xi)$ 
quand $t \tend \minfty$, voir  \cite[Prop. II.9.8]{BrHa}.
 Notons  $\ell=\ell(g)$, et $x=r(0)$ et $y=r(\ell)$.
En faisant agir $g^n$ 
pour  $n\in\ZZ$
 on obtient que
$\angle_{x}(\xim,\xi)=\angle_{y}(\xim,\xi)=\angleTits(\xim,\xi)$
et $\angleTits(\xi,\xipp)=\angle_{x}(\xi,\xipp)$.
Or 
$\angle_{x}(\xim,\xi) + \angle_{x}(\xi,\xipp)
\geq \pi$ et
$\angle_{x}(\xi,y)+\angle_{y}(x,\xi) \leq \pi$, ce qui conclut.
\end{preuve}

\begin{defi}[Points attractif et r{\'e}pulsif]
\label{defi- pf attractif} 
  On dira que $g\in\Isom(\Es)$ a des {\em points fixes
    attractif} et {\em r{\'e}pulsif} {\em {\`a} l'infini} s'il existe
  $\xippg,\ximg\in\bordinf\Es$  oppos{\'e}s (n{\'e}\-c{\'e}ssai\-re\-ment uniques) 
  tels que pour un (tout) $x\in\Es$ on ait $g^k x \tend \xippg$ et
  $g^{-k} x \tend \ximg$ quand $k \tend \pinfty$. Alors $\xippg$ et
  $\ximg$ sont clairement fix{\'e}s par $g$.
\end{defi}

\begin{rema}
\label{rema- pf attractif ES et IMM}
Si $g$ est axiale, alors  les extr{\'e}mit{\'e}s d'un axe de $g$ conviennent.   
Si $\Es$ est un  immeuble affine complet, ils existent donc d{\`e}s que
$\ell(g)>0$ par le th{\'e}or{\`e}me \ref{theo- immob}. On verra que c'est
{\'e}galement le cas lorsque $\Es$ est un espace sym{\'e}trique de type non compact (proposition \ref{prop- d_g(x) si x dans Min(h)}).
\end{rema}

\begin{prop}
\label{prop- critere pf attractif} 
Soit $g\in\Isom(\Es)$  fixant deux points oppos{\'e}s $\xim$ et $\xipp$
dans $\bordinf\Es$. On note $g=(\gp, t)$ la d{\'e}composition de $g$
sur $F=\Fximp=C\times \RR$.

Si on a $t=\ell(g)$,
ou, de mani{\`e}re {\'e}quivalente par (\ref{eq- faisceau stable et ell(g)}), si $t>0$ et $\ell(\gp)=0$, 
alors $g$ a des points fixes attractif et r{\'e}pulsif {\`a} l'infini, et ce
sont $\xipp$ et $\xim$.
\end{prop}

\begin{preuve}
Soit $x=(\xp,0)$ dans $F$.
Alors 
\( \frac{1}{kt} d( g^k x, \; (\xp,kt) ) 
=  \frac{1}{kt} d(\xp,    \; (\gp)^k \xp ) \),
qui tend vers $0$ car $\ell(\gp)=0$ par (\ref{eq- d(x, g^k x)}).
Donc $g^k x \tend \xipp$ et $g^{-k} x \tend \xim$ quand $k\tend\pinfty$.
\end{preuve}

\begin{prop}
\label{prop- dynamique}
Soit $g$ ayant des points fixes attractif et r{\'e}pulsif $\xipp$ et
$\xim$ {\`a} l'infini.

Pour tout $x\in\Fg$
et tout $D>0$, pour tout voisinage $V$
de $\xipp$ dans $\bordinf\Es$
on a pour $k$ suffisamment grand
$$g^k(\OmbxyD \xim x D)\subset V$$
En particulier, $\xipp$ est le seul point fixe
de $g$ dans $\bordinf\Es$ qui est oppos{\'e} {\`a} $\xim$.
\end{prop}

\begin{preuve}
 Soit $x\in \Fximp$ et
$D>0$. Soit $V$ un voisinage de $\xipp$ dans $\bordinf\Es$.
Supposons par l'absurde qu'il existe une suite $(\xik)_k$ dans
$\OmbxyD \xim x D$
 et une suite croissante $(\nk)_k$ d'entiers avec $\xipk=
g^{\nk}\xik \notin V$ pour tout $k$.
Soit $\yk$ dans $F_{\xim \xik}$ tel que $d(x,\yk)\leq D$.  
Comme $g^{\nk} x \tend \xipp$,  on a  aussi
$g^{\nk} \yk \tend \xipp$, car $d(g^{\nk} \yk,g^{\nk} x)\leq D$. 
Par cons{\'e}quent $\angle_x(\xipp, g^{\nk} \yk) \tend 0$
et $\angle_x(\xim, g^{\nk}
\yk) \tend \pi$.
On en d{\'e}duit
que $\angle_{g^{\nk} \yk}(x,\xim) \tend 0$, puis 
que $\angle_{g^{\nk} \yk}(x, \xipk) \tend \pi$ (car
$\angle_{g^{\nk} \yk}(\xim, \xipk)=\pi$), puis que $\angle_x(g^{\nk}
\yk, \xipk) \tend 0$.
Comme
$\angle_x(\xipp, \xipk)
\leq \angle_x(\xipp,g^{\nk} \yk)+\angle_x(g^{\nk}\yk, \xipk)$,
on a $\angle_x(\xipp, \xipk)\tend 0$ 
et $\xipk \tend \xipp$, contradiction.
\end{preuve}

\begin{prop}
\label{prop- faisceau stable et point attractif}
  Soit $g\in\Aut(\Es)$ fixant deux points oppos{\'e}s
$\xipp,\xim\in\bordinf\Es$. Notons $(\gp, t)$ la d{\'e}composition de $g$
sur $F=\Fximp=C\times \RR$.
On suppose que $g$ poss{\`e}de des points
fixes attractif et r{\'e}pulsif $\xippg$ et $\ximg$  dans $\bordinf\Es$.
Alors $\xippg,\ximg\in\bordinf F$ et 
\[\tan(\angleTits(\xippg,\xipp)) \leq \frac{\ell(\gp)}{t}\]
\end{prop}

\begin{preuve}
Soit $x\in F$. Comme  
$\xippg=\lim_{n\tend\pinfty} g^nx$ et $g(F)=F$, on
a $\xippg\in \bordinf F$.
De plus 
$\angleTits(\xippg,\xipp)=
\angle_x(\xippg,\xipp)=\lim_{n\tend\pinfty}\angle_x(g^nx,\xipp)$
et pour $x=(\xp,0)$ on a 
$\tan(\angle_x(g^n x,\xipp))=\frac{1}{nt}d(\xp,(\gp)^n\xp)$.
Or  $\lim_{n\tend\pinfty}\frac{1}{n}d(\xp,(\gp)^n\xp)\leq \ell(\gp)$
par (\ref{eq- d(x, g^k x)}).
\end{preuve}

\subsection{Espaces sym{\'e}triques et immeubles affines}
\label{ss- prel ES IMM}

\subsubsection{G{\'e}n{\'e}ralit{\'e}s}

\paragraph{Espaces sym{\'e}triques.}
\label{ss- prel ES}
Les espaces sym{\'e}triques consid{\'e}r{\'e}s dans cet article sont
 sans facteurs compacts.
On renvoie {\`a} \cite{Hel} et \cite{Ebe} comme r{\'e}f{\'e}rences g{\'e}n{\'e}rales. 
Soit $\Es$ un espace sym{\'e}trique. 
Rappelons que $\Es$ est une vari{\'e}t{\'e} riemannienne 
compl{\`e}te simplement connexe, {\`a} courbure sectionnelle n{\'e}gative ou
nulle. 
En particulier $\Es$ est un espace m{\'e}trique $\CAT(0)$ complet et propre,
et on reprendra les notations de la section \ref{ss- CAT(0)}.

On note $\Es_0$ le facteur euclidien de $\Es$ et  $\Es=\Es_0 \times
\Es'$ la d{\'e}composition correspondante.
On appellera groupe des {\em automorphismes} de $\Es$ et on notera
$G=\Aut(\Es)$ la composante neutre du sous-groupe des isom{\'e}tries de $\Es$
agissant par translation sur $\Es_0$.
Le groupe $G$ agit proprement et transitivement sur $\Es$ et c'est un
groupe de Lie r{\'e}el. 
On a $G=G_0\times G'$ o{\`u}
$G_0=\Aut(\Es_0)=(\RR^k,+)$ (avec $k=\dim\Es_0$) et
$G'=\Aut(\Es')$ est semisimple de centre trivial. 

Si $r$ est une g{\'e}od{\'e}sique de $\Es$ alors $\Fr$ et $\Cr$ sont des
sous-espaces totalement g{\'e}od{\'e}siques de $\Es$.

Si $\Es'$ est un sous-espace totalement g{\'e}od{\'e}sique de $\Es$, alors
$\Es'$ est un espace sym{\'e}trique sans facteur compact, 
et 
la composante neutre de $\Stab_G(\Es')$
agit sur $\Es'$ par automorphismes, transitivement.
\label{def plats max marques ES}
On fixe un {\em plat} (sous-espace totalement g{\'e}od{\'e}sique euclidien) maximal
 (pour l'inclusion)  $\Aa$ de $\Es$. 
 Les restrictions {\`a} $\Aa$ des {\'e}l{\'e}ments de $G$ seront appel{\'e}s {\em plats
 maximaux marqu{\'e}s} de $\Es$.

\paragraph{Immeubles affines.}
Les immeubles affines consid{\'e}r{\'e}s dans cet article sont m{\'e}triques, pas
n{\'e}\-ces\-sai\-re\-ment localement compacts, ni discrets, c'est-{\`a}-dire
qu'ils n'admettent pas n{\'e}\-ces\-sai\-re\-ment de structure de complexe
(poly)simplicial \cite{Tits}.
Pour leur d{\'e}\-fi\-ni\-tion, on renvoie {\`a} \cite{ParImm}, o{\`u} elle est
{\'e}tudi{\'e}e en d{\'e}tail.
On reprendra le vocabulaire et les notations de cet article, auquel on
renvoie {\'e}galement pour les propri{\'e}t{\'e}s utilis{\'e}es et non
d{\'e}montr{\'e}es ci-dessous (notamment pour l'{\'e}quivalence avec la
d{\'e}finition de Kleiner-Leeb \cite{KlLe}, dont on aura besoin pour le
th{\'e}or{\`e}me \ref{theo- cone as = immeuble}).

On se fixe un  espace vectoriel $\Aa$ euclidien de dimension finie
et un groupe de r{\'e}flexions fini $\Wvect$ agissant sur $\Aa$ (c'est-{\`a}-dire un
sous-groupe fini d'isom{\'e}tries vectorielles de $\Aa$, engendr{\'e} par des
r{\'e}flexions). 
 On fixe {\'e}galement une chambre de Weyl ferm{\'e}e $\Cb$ de $\Aa$.
Soit $\Waff$ un 
 groupe de r{\'e}flexions affine de type $(\Aa, \Wvect)$,
 c'est-{\`a}-dire un
sous-groupe d'isom{\'e}tries affines de $\Aa$, de projection
vectorielle $\Wvect$, tel qu'il existe $a$ dans $\Aa$ tel que $\Waff=T\Waff_a$,
o{\`u} $\Waff_a$ fixe $a$ et $T$ est un sous-groupe de translations de $\Aa$
(ou, de mani{\`e}re {\'e}quivalente, engendr{\'e} par des r{\'e}flexions affines).

Soit $\Es$ un immeuble affine, 
model{\'e} sur $(\Aa, \Waff)$.
On note $\A$ son syt{\`e}me d'appartements marqu{\'e}s aussi appel{\'e}s 
{\em plats maximaux marqu{\'e}s}.
\label{def plats max marques IMM}
Un {\em automorphisme} de $\Es$ est une bijection de
$\Es$ dans lui-m{\^e}me qui pr{\'e}serve $\A$.
On rappelle que $(\Es,d)$ est un espace m{\'e}trique $\CAT(0)$ (pas
n{\'e}c{\'e}ssairement complet).
\paragraph{Sous-immeubles.}
\label{ss- sous-immeuble}
On dira qu'un groupe de r{\'e}flexions $(\Aa',\Waff')$ 
est un {\em sous-groupe de r{\'e}flexions}
de $(\Aa, \Waff)$ 
si $\Aa'$ est un sous-espace affine 
de $\Aa$
et les {\'e}l{\'e}ments de $\Waff'$ 
sont les restrictions {\`a} $\Aa'$ d'{\'e}l{\'e}ments de
$\Waff$ 
stabilisant $\Aa'$.
On dira alors qu'un immeuble $(\Es',\A')$ model{\'e} sur $(\Aa',\Waff')$
est un
{\em sous-immeuble} de $\Es$ si 
$\Es'\subset \Es$ 
et si les appartements marqu{\'e}s de $\Es'$ sont les restrictions {\`a}
$\Aa'$ d'appartements marqu{\'e}s de $\Es$.  
Alors $\Es'$ est un sous-espace convexe de $(\Es,d)$.
Par exemple, si $\Es$ est un produit d'immeubles $(\Es_i,\A_i)$ 
model{\'e}s sur $(\Aa_i, \Waff_i)$, 
on v{\'e}rifie facilement que
les facteurs sont des sous-immeubles de $\Es$.
Un autre exemple de base est le suivant. Soit $r$ une
g{\'e}od{\'e}sique de $\Es$, incluse dans un appartement (ce qui est toujours
le cas si $\A$ est le syst{\`e}me d'appartements maximal) et $\Es' =\Fr$. 
Alors $\Es'$ est la r{\'e}union des appartements de $\Es$ contenant $r(\pinfty)$
et $r(\minfty)$ dans leur bord {\`a} l'infini.
Soit $u\in \Aa$ un vecteur 
tel que $r$ est {\'e}gale {\`a} $t\mapsto
 f(a+tu)$ pour un plat maximal marqu{\'e} $f:\Aa\fleche \Es$ de $\Es$ et un
 $a\in \Aa$.
Soit $\Waff'$ le sous-groupe des {\'e}l{\'e}ments
de $\Waff$ fixant (vectoriellement) $u$.   
On peut d{\'e}montrer (par exemple en suivant
\cite[prop. 4.8.1]{KlLe}) que $\Es'$ est un sous-immeuble de $\Es$
model{\'e} sur $(\Aa,\Waff')$.
Les plats maximaux marqu{\'e}s de $\Es'$ sont les plats maximaux marqu{\'e}s
$f:\Aa\fleche \Es$ de $\Es$ tels que $t\mapsto f(tu)$ est une
g{\'e}od{\'e}sique parall{\`e}le {\`a} $r$.

\subsubsection{Espace associ{\'e} {\`a} un groupe r{\'e}ductif}
\label{ss- espace associe a un GR}

Soit $\KK$ un corps valu{\'e}, ou bien {\'e}gal {\`a} $\RR$ ou $\CC$, ou bien
ultram{\'e}trique.

\paragraph{Groupe r{\'e}ductif sur $\KK$.}
 Soit $\Gb$ un groupe alg{\'e}brique lin{\'e}aire connexe r{\'e}ductif d{\'e}fini sur
$\KK$.
On peut supposer que $\Gb$ est un sous-groupe
alg{\'e}brique d{\'e}fini sur $\KK$ de $\SLn$.

Dans le cas o{\`u} $\KK=\RR$, on suppose que 
 $G$ est un sous-groupe ferm{\'e} de $\Gb(\RR)$ contenant sa composante
neutre $\Gb(\RR)^0$.
Si $\KK\neq \RR$, on suppose  que $G=\Gb(\KK)$
Le groupe $G$ est muni de la topologie induite par celle de $\KK$.
C'est 
un groupe topologique m{\'e}trisable, {\`a} base d{\'e}nombrable, localement
compact si $\KK$ l'est.

\paragraph{Espace sym{\'e}trique associ{\'e} {\`a} un groupe r{\'e}ductif archim{\'e}dien.}
\label{ss- espace sym associe a un groupe reductif reel}
On suppose que 
$\KK=\RR$ ou $\CC$ (notons que, si $\KK=\CC$ le groupe $G$ est le
groupe des points r{\'e}els du $\RR$-groupe r{\'e}ductif connexe $\Hb$ obtenu
{\`a} partir de $\Gb$ par restriction des scalaires).
Soit $K$ un sous-groupe  compact maximal de $G$.
Alors  l'espace homog{\`e}ne $\Es=G/K$
est un espace sym{\'e}trique sans facteur compact, dit {\em associ{\'e}} {\`a}
$G$, pour toute m{\'e}trique riemannienne $G$-invariante sur $\Es$
(cela d{\'e}coule de l'existence d'une involution de Cartan de $(G,K)$,
voir \cite[24.6(a)]{Borel} et \cite{Hel}).
Le groupe $G$ agit proprement, transitivement, par automorphismes,
sur $\Es$. 
\paragraph{Immeuble d'un groupe r{\'e}ductif non archim{\'e}dien \cite{BrTi,BrTi2}.}
\label{ss- immeuble d'un groupe reductif sur un corps non arch}

On renvoie par exemple {\`a} \cite{RousseauEe04} pour plus de d{\'e}tails.
On suppose que $\KK$ est ultram{\'e}trique
et que les hypoth{\`e}ses de 
\cite[4 ou 5]{BrTi2} sont satisfaites.
Rappelons que 
c'est notamment le cas dans les deux situations suivantes :

- $G$ est d{\'e}ploy{\'e} sur le corps $\KK$ (par exemple $\Gb=\SLn$). 

- $\KK$ est {\`a} valuation discr{\`e}te, hens{\'e}lien (par exemple
  complet),  de corps r{\'e}siduel parfait (par exemple fini).
On peut alors construire l'immeuble de Bruhat-Tits ({\'e}largi) $\Es=\mathcal{B}(\Gb,\KK)$ de
$\Gb$ sur $\KK$.
C'est un immeuble affine, pas n{\'e}cessairement discret ni complet.
Le groupe $G$ agit contin{\^u}ment, par automorphismes de $\Es$, 
\label{SLnK agit transitivement sur les
appartements marques de son imm de BT}
transitivement sur les appartements marqu{\'e}s,
cocompactement sur $\Es$. 
Si $\KK$ est un corps local, alors $\Es$ est propre et l'action de $G$ sur $\Es$ est propre \cite[2.2]{Tits2}.  

\subsubsection{Type des segments}
\label{ss- type}

On suppose d{\'e}sormais que $\Es$ est un espace sym{\'e}trique
ou un immeuble affine (voir section \ref{ss- prel ES  IMM}) et on note
$G=\Aut(\Es)$. 
On fixe 
une chambre de Weyl ferm{\'e}e $\Cb$ du plat mod{\`e}le $\Aa$.
C'est un domaine fondamental strict pour l'action 
du  groupe de Weyl (vectoriel) $\Wvect$ sur $\Aa$. 
On note $\type:\Aa\fleche\Cb$ la projection 
 correspondante, qui {\`a} un vecteur dans $\Aa$ associe son
{\em type} dans $\Cb$.
On note $v^\opp=\type(-v)$ type {\em oppos{\'e}} {\`a} $v\in\Cb$.
Si $x,y\in\Es$, il existe un unique vecteur
de $\Cb$, appel{\'e} le {\em type} du segment $xy$ (ou la {\em \Cdistance}
de $x$ {\`a} $y$) et not{\'e} 
$\Cd(x,y)$ ou $\Cd xy$,
tel qu'il existe un plat maximal marqu{\'e}
 $f:\Aa\fleche \Es$, et deux
points $A$ et $B$ de $\Aa$, avec $f(A)=x$, $f(B)=y$ et
$\overrightarrow{AB}=\Cd(x,y)$.
On a clairement $\norme{\Cd(x,y)}=d(x,y)$ et 
$\Cd(y,x) =\Cd(x,y)^\opp$ pour tous $x,y$ de $\Es$.  
\label{ss- type - props elem}
Les automorphismes de $\Es$ pr{\'e}servent le type des segments.
\label{ss- type - les autom preserve le type}
Dans le cas des espaces sym{\'e}triques $G$ agit transitivement
sur les segments de type donn{\'e}.

\begin{rema}
Lorsque $\Es$ est l'espace associ{\'e} {\`a} un groupe 
r{\'e}\-duc\-tif $G$ sur un corps local 
(voir section \ref{ss- espace associe a un GR}),
si $K$ est un sous-groupe compact maximal
de $G$ et $\xo\in\Es$ est le point fix{\'e} par $K$, 
alors, pour $g$ dans $G$, le vecteur $\Cd(\xo,g\xo)\in\Cb$
s'identifie {\`a} la projection de Cartan de $g$ associ{\'e}e {\`a} $K$
(voir par exemple \cite[1.1 et 2.3]{Benoist}).
\end{rema}

On renvoie {\`a} \cite{ParType} pour de plus fines propri{\'e}t{\'e}s du type des
segments. On utilisera la propri{\'e}t{\'e} suivante.

\begin{prop} 
\label{prop- Cd 1-Lip}
\cite{ParType}
Le type est $1$-Lipschitz en chaque coordonn{\'e}e, plus pr{\'e}\-ci\-s{\'e}\-ment
pour tous $x,y,\xp,\yp$ de $\Es$, on a
$$\norme{\Cd(x,y)-\Cd(\xp,\yp)}\leq d(x,\xp)+d(y,\yp)$$ 
En particulier,
 $\Cd : \Es\times \Es \fleche\Cb$ est continue 
\qed\end{prop}

Etant donn{\'e} $g\in\Aut(\Es)$,
on note $\Cd_g$
l'application continue 
$$\Cd_g :
\begin{array}{ccl}
\Es &\fleche &\Cb\\
x&\mapsto&\Cd(x,gx)
\end{array}$$
On utilisera la propri{\'e}t{\'e} suivante, qui d{\'e}coule 
de la proposition \ref{prop- Cd 1-Lip}.

\begin{prop} \label{prop- type_gk cv vers type_g}
Soit $\gk\tend g$  dans $\Aut(\Es)$. 
Alors $\Cd_{\gk}\tend \Cd_g$ uniform{\'e}ment  sur les
born{\'e}s.  \qed \end{prop}

\paragraph{Propri{\'e}t{\'e}s fonctorielles.}
\label{ss- type - produit, sous-espace}
Si $\Es$ est euclidien, alors $\Aa=\Es$, $\Wvect=\{\id\}$, $\Cb=\Aa=\Es$, et
$\Cd(x,y) = \overrightarrow{xy}$ pour tous $x$, $y$ de $\Es$.
Si $\Es=\Es_1\times \Es_2$ est un produit de deux espaces sym{\'e}triques ou
immeubles affines $\Es_i$.
Alors $\Aa=\Aa_1\oplus \Aa_2$ o{\`u}  $\Aa_i$ est un  plat maximal de $\Es_i$,
$\Wvect=\Wvect_1\times \Wvect_2$ o{\`u}  $\Wvect_i$ est le groupe de Weyl de $\Aa_i$,
et $\Cb=\Cb_1+\Cb_2$ o{\`u}  $\Cb_i$ est une chambre de Weyl ferm{\'e}e de
$\Aa_i$.
Les plats  maximaux marqu{\'e}s $f:\Aa \fleche \Es$ sont de la forme
$f=(f_1,f_2)$ o{\`u} $f_i:\Aa_i \fleche \Es_i$ est un  plat  maximal marqu{\'e}.
Si $\Cd_i:\Es_i\times \Es_i\fleche\Cb_i$ est le type des segments de $\Es_i$, on a

\begin{equation}
  \Cd (x,y) =\Cd_1(x_1,y_1)+\Cd_2(x_2,y_2)
\end{equation}
 pour tous $x=(x_1,x_2)$ et $y=(y_1,y_2)$ de $\Es$.
Soit $\Es'$ un sous-espace sym{\'e}trique ou un sous-immeuble de $\Es$.
Soit $\Aa'$ un plat maximal de $\Es'$, et $\Wvect'$ son groupe de
Weyl, et $\Cb'$ une chambre de Weyl ferm{\'e}e de $\Aa'$.
Alors $\Aa'$ s'identifie {\`a} un sous-espace totalement g{\'e}od{\'e}sique de $\Aa$, 
et les {\'e}lements de $\Wvect'$ sont des restrictions {\`a} $\Aa'$ d'{\'e}l{\'e}ments de
$\Wvect$ stabilisant $\Aa'$.
Par cons{\'e}quent, si $\Cd':\Es'\times \Es'\fleche\Cb'$ est le type des
segments de $\Es'$, on a 
\begin{equation}
  \Cd =\type \circ \Cd'
\end{equation}
 en restriction {\`a} $\Es'\times \Es'$.

\paragraph{Type de direction.}
Notons $\bord \Aa$ l'ensemble des vecteurs unitaires de $\Aa$, identifi{\'e} {\`a} l'ensemble des demi-droites de $\Aa$, 
muni de la distance angulaire,
et $\bord : \Aa\fleche \bord \Aa$ la projection correspondante.
Pour  $\untype,\untype'\in\bCb$ (qui est l'ensemble
not{\'e} $\Dmod$ dans \cite{KlLe}), on note $D(\untype,\untype')$
l'ensemble fini des angles possibles entre $v$
et $v'$ de $\bord\Aa$ de types respectifs $\untype$ et $\untype'$.

Si $\Es$ est un immeuble affine, alors il y a rigidit{\'e} des
angles \cite[4.1.2]{KlLe}: pour tous $x,y,z\in\Es$
avec $y,z\neq x$ on a $\angle_x(y,z)\in D(\untype,\untype')$
avec $\untype=\bord\Cd(x,y)$ et $\untype'=\bord\Cd(x,z)$.
 
Le {\em type (de direction)} $\btype r$ d'une g{\'e}od{\'e}sique (ou d'un
rayon) $r$ est  $\bord\Cd( r(s), r(t) )$ pour tout $s < t$.  Deux
g{\'e}od{\'e}siques parall{\`e}les ont m{\^e}me type.  
 Les automorphismes de $\Es$ pr{\'e}servent le type des g{\'e}od{\'e}siques.

Dans le cas des espaces sym{\'e}triques, le stabilisateur $\Stab_G(x)$
dans $G$ de $x\in\Es$ agit transitivement sur les g{\'e}od{\'e}siques de type
donn{\'e} issues de $x$.

\subsubsection{Structure du bord {\`a} l'infini}
On suppose  que $\Es$ est un espace sym{\'e}trique ou un immeuble
affine complet.
 
\paragraph{Type d'un point {\`a} l'infini.}
Deux rayons g{\'e}od{\'e}siques asymptotes ont m{\^e}me type.
 On peut donc d{\'e}finir le type (not{\'e} {\'e}galement $\btype$) dans $\bCb$
des points de $\bordinf\Es$. 
 Notons $(\bordinf\Es)_\untype$  le sous-ensemble des points de
 $\bordinf\Es$ de type $\untype\in\bord\Cb$.

L'action de $\Aut(\Es)$ sur $\bordinf\Es$ pr{\'e}serve le type des points.

Dans le cas des espaces sym{\'e}triques l'action de $G$ sur $\bordinf\Es$
est transitive sur $(\bordinf\Es)_\untype$ pour tout $\untype\in\bord\Cb$,
et si $\xi\in\bordinf\Es$, alors $(\bordinf\Es)_\untype=G\xi$
s'identifie {\`a} la vari{\'e}t{\'e}
$G/P$, o{\`u} $P$ est le sous-groupe parabolique
$\Stab_G(\xi)$ de $G$.
\label{ss- prel ES - Gxi variete}
\label{prel-strure d'immeuble sur le bord}
Si $\Es$ est complet, son bord $\bordinf\Es$ muni de la distance de
Tits (donn{\'e}e par $\angleTits$) est isom{\'e}trique {\`a} la r{\'e}alisation
g{\'e}om{\'e}trique model{\'e}e sur $(\bord \Aa,\Wvect)$ \cite[section
  2.6]{ParImm} d'un immeuble sph{\'e}rique.
Les chambres de cet immeuble sont les bords des chambres de Weyl de
$\Es$, et ses appartements sont les bords des plats maximaux de $\Es$.
Si $\xi,\xi'\in\bordinf \Es$, il existe un plat maximal $A$ de $\Es$
tel que $\xi,\xi'\in\bordinf A$, voir \cite[2.21.14]{Ebe}, \cite[1.2,
  (A4)]{ParImm}.
Pour tout $x\in A$, on a
$\angle_x(\xi,\xi')=\angleTits(\xi,\xi') \in D(\btype (\xi),\btype (\xi'))$.
En particulier si $\angleTits(\xi,\xi')=\pi$, alors $\xi$ et $\xi'$
sont oppos{\'e}s.

Dans le cas des espaces sym{\'e}triques $\Aut(\Es)$ agit transitivement
sur les couples de points du bord oppos{\'e}s, de types donn{\'e}s.

\paragraph{Points visuellement presque oppos{\'e}s.}
Si $\Es$ est un immeuble affine, si $\xim,\xipp\in\bordinf
\Es$ sont visuellement presque oppos{\'e}s en $x\in\Es$, alors ils sont
effectivement oppos{\'e}s en $x$ : plus pr{\'e}cis{\'e}ment, en notant $\untype$ et
$\untype^\opp$ les types respectifs de $\xim$ et $\xipp$, si
$\angle_x(\xim, \xipp)> \max( D(\untype,\untype^\opp) - \{\pi\})$
alors $x\in\Fximp$ (par rigidit{\'e} des angles).

Dans les espaces sym{\'e}triques on a la propri{\'e}t{\'e} analogue plus faible
suivante.

\begin{prop} 
\label{prop- points visuellement presque opposes} 
Soit $\untype\in\bCb$ et $\untype^\opp$ le type oppos{\'e}. 
Pour tout $D>0$, il existe $\eps>0$ tel que pour tous
$\xim,\xipp\in\bordinf \Es$ de types respectifs $\untype$ et
$\untype^\opp$,
et pour tout $x\in\Es$, 
si $\angle_x(\xim,\xipp)\geq \pi-\eps$ alors $d(x,\;\Fximp)\leq D$.
\end{prop}

\begin{preuve} 
Sinon il existe $D>0$ et pour tout $k\in\NN$ un point $\xk\in\Es$ et
$\xippk,\ximk$ dans $\bordinf\Es$ de types respectifs $\untype$ et
$\untype^\opp$ tels que $\angle_{\xk}(\ximk,\xippk)\tend \pi$ et
$d(\xk,F_{\ximk\xippk})\geq D$.
Pour $k$ assez grand, $F_{\ximk\xippk}$ est non vide (car
$\angle_x(\xim, \xipp) > \max ( D(\untype,\untype^\opp) - \{\pi\})$
implique  $\angle_T(\xim,\xipp) = \pi$).
Soit $\yk$ la projection de $\xk$ sur $F_{\ximk\xippk}$. 
Quitte {\`a} appliquer des automorphismes, on peut supposer
que les suites  $\yk$, $\ximk$ et $\xippk$ sont constantes,
respectivement {\'e}gales {\`a} $y$, $\xim$ et
$\xipp$.
Soit $\zk\in [\xk, y]$ tel que $d(\zk,y) = D$.
On a $\angle_{\xk}(y,\xim)\leq \angle_{\zk}(y,\xim)\leq {\pi\over 2}$ et
$\angle_{\xk}(y,\xipp)\leq \angle_{\zk}(y,\xipp)\leq {\pi\over 2}$.
Comme $\angle_{\xk}(\xim,\xipp)\leq
\angle_{\xk}(y,\xim)+\angle_{\xk}(y,\xipp)$, on a
donc $\angle_{\zk}(y,\xim)\tend{\pi\over 2}$ ainsi que 
$\angle_{\zk}(y,\xipp)\tend{\pi\over 2}$.
Quitte {\`a} extraire on a 
$\zk\tend z$ avec $d(z,\Fximp )= D>0$ et
$\angle_{z}(\xim,\xipp)=\pi$, impossible.
\end{preuve}

\paragraph{Ombres vues de l'infini.}
\begin{prop} 
\label{prop- ombres forment base de vois} 
Soit $\xi\in\bordinf\Es$ et $r$ une g{\'e}od{\'e}sique telle que
$r(\pinfty)=\xi$. Notons  $\xim=r(\minfty)$. Les ombres 
$\OmbxyD {\xim} {r(t)} D$, avec $t\in\RR$ et $D>0$, forment une base de
voisinages de $\xi$ dans $(\bordinf\Es)_\untype$, o{\`u} $\untype$ est le
type de $\xi$.
\end{prop}

\begin{preuve} 
On a clairement toujours $\OmbxyD {\xim} {r(t)} D \subset
(\bordinf\Es)_\untype$.
Soit $x=r(0)$.  
Les ombres $\OmbxyD {x} {r(t)} D$, avec
$t\in\RR^+$ et $D>0$, forment une base de voisinages de $\xi$ dans
$\bordinf \Es$, par d{\'e}finition de la topologie des c{\^o}nes.
Soient $t\in\RRpet$ et $D>0$.  On v{\'e}rifie facilement que $\OmbxyD
{\xim} {r(t)} {\;\frac{1}{2}D} \subset \OmbxyD {x} {r(t)} D$.  
Soit $\eps>0$ donn{\'e} par la proposition 
\ref{prop- points visuellement presque opposes}. 
Soit $0<D'\leq \tan(\eps)\; t$. Alors $ \OmbxyD {x} {r(t)} {D'}\cap
(\bordinf\Es)_\untype$ est inclus dans $\OmbxyD {\xim} {r(t)} D$.
En effet si $\xi'$ est dans $\OmbxyD {x} {r(t)} {D'} \cap G\xi$, alors
\[\angle_x(r(t),\xi') \leq \angletilde_x(r(t),\xi') \leq \arctan
(D'/t)\leq \eps\] 
D'apr{\`e}s la proposition \ref{prop-
  points visuellement presque opposes}, comme $\xi'$ est de m{\^e}me type
que $\xi$ et $\xim$ de type oppos{\'e}, on a alors $d(x,F_{\xim \xi'})\leq
D$, c'est-{\`a}-dire $\xi'\in \OmbxyD {x} {r(t)} D$.
\end{preuve}

\subsubsection{Isom{\'e}tries}
\label{ss- prel isom}
Dans les immeubles affines complets, on a le r{\'e}sultat suivant (qu'on
utilisera de mani{\`e}re cruciale en section \ref{s- VdT}).

\begin{theo}{\cite[Th{\'e}or{\`e}me 4.1]{ParImm})}
\label{theo- immob} 
Soit $\Es$ un immeuble affine complet. Il existe
$K>0$ tel que pour tout $g\in\Aut(\Es)$ et pour tout
$x\in \Es$, on a $d(x, \Min(g))\leq K d(x,gx)$. En particulier $g$ est
semisimple.  \end{theo}

\medskip

Dans les espaces sym{\'e}triques, la situation n'est pas si simple mais
la d{\'e}\-com\-po\-si\-tion de Jordan va nous permettre d'{\'e}tablir quelques
propri{\'e}t{\'e}s 
dont on aura besoin en section \ref{s- VdT}.

\bigskip

On suppose d{\'e}sormais que $\Es$ est un espace sym{\'e}trique.

\paragraph{D{\'e}composition de Jordan.}
On renvoie {\`a} \cite[2.19]{Ebe} pour plus de d{\'e}tails et pour le rapport avec
la d{\'e}composition de Jordan usuelle dans le groupe lin{\'e}aire, via la
repr{\'e}sentation adjointe de $G=\Aut(\Es)$.

\label{ss- transvections}
Si $r$ est une g{\'e}od{\'e}sique de $\Es$ et $t\in \RR$, il existe 
$g\in G$, appel{\'e} {\em transvection} le long de $r$
de longueur $t$, r{\'e}alisant le transport parall{\`e}le le long de $r$ de
$r(s)$ {\`a} $r(s+t)$, voir \cite[IV, thm 3.3 et remarque]{Hel}.
Si $t\neq 0$ on a $\Min(g)=\Fr$.
De plus \cite[Prop. 4.1.4]{Ebe} le centralisateur $Z_G(g)$ de $g$ dans
$G$ agit transitivement sur $\Min(g)$.
On dit que $g\in G$ est {\em hyperbolique} si $g$ est une transvection
\cite[2.19.21]{Ebe}.

On dira  que $g\in G$ est {\em unipotent} 
si l'adh{\'e}rence de sa classe de conjugaison dans $G$
contient l'identit{\'e} \cite[2.19.27]{Ebe}.
Alors on a clairement $\ell(g)=0$.

\label{ss- prel ES - dec de jordan}
Pour tout $g\in G$, il existe  une  unique d{\'e}composition
$g=heu$ dans $G$, appel{\'e}e la {\em d{\'e}composition de Jordan} de $g$, telle que
 $e$, $h$ et $u$ sont respectivement elliptique, hyperbolique et
unipotent, et commutent deux {\`a} deux \cite[2.19.24]{Ebe}.
On v{\'e}rifie ais{\'e}ment que si $h=\id$ alors $\ell(g)=0$.

\paragraph{Isom{\'e}tries des espaces sym{\'e}triques.}
Soit  
$g\in \Aut(\Es)$. 
Soit $g=heu$ la d{\'e}composition de Jordan de $g$.
Notons $\phi=eu$ et $\Fg=\Min(h)$, qui est stable par $h$,  $e$ et $u$.
Si $\ell(g)>0$ on note $r$ une g{\'e}od{\'e}sique translat{\'e}e par $h$ (alors
$\Fg=\Fr$).

\begin{prop}
\label{prop- d_g(x) si x dans Min(h)}
\label{prop- prel ES - l(g)=l(h)}
On suppose que $\ell(g)>0$.
En restriction {\`a} $\Fg=\Cr \times r$ 
on a $h = (\id,\ell(g))$, $\phi=(\phi, 0)$ et $g=(\phi,\ell(g))$.
En particulier $\ell(g)=\ell(h)$
et $g$ poss{\`e}de  des points fixes attractif et r{\'e}pulsif {\`a} l'infini, qui
sont les extr{\'e}mit{\'e}s de $r$.
\end{prop}

\begin{preuve}
Sur $\Min(h)=\Fr$ on a  $h=(\id,\ell(h))$ avec $\ell(h)>0$,
$\phi=(\phi',t)$.
On a alors $0=\ell(\phi)=\sqrt{\ell(\phi')^2+t^2}$, donc $t=0$ et
$\phi'=\phi$ sur $C_r$.  
On a donc $g=(\phi,\ell(h))$ sur $\Fg$, 
donc $\ell(g)=\sqrt{\ell(\phi)^2+\ell(h)^2}=\ell(h)$.
On conclut par la proposition \ref{prop- critere pf attractif}.
\end{preuve}

On aura besoin en section \ref{s- VdT} de contr{\^o}ler la distance {\`a}
$\Min(h)$ en fonction du d{\'e}placement, avec le lemme suivant.

\begin{lemm} 
\label{lemm- d_g et distance a F_g}
Pour tout $\eps> 0$ il existe $\eta >0$ tel que si $x\in\Es$ v{\'e}rifie
$d_g(x) \leq \ell(g)+\eta$, et $y$ est la projection de $x$ sur 
$\Fg$, alors

\begin{enumerate}

\item  \label{item- d(gy,hy)}
$d(gy,hy)\leq \eps$.

\item \label{item- d(x,Min(h))}
 $d(x,\Min(h)) \leq \eps$

\end{enumerate}
\end{lemm}

\begin{preuve}
C'est trivialement vrai si $\ell(g)=0$, car alors $h=\id$ et
$\Fg=\Min(h)=\Es$. 
Supposons que $\ell(g)>0$, et soit $\eps>0$.
Le point \ref{item- d(gy,hy)} 
d{\'e}coule de $d_g(y)\leq d_g(x)$ et $d_{g}(y)=\sqrt{\ell(g)^2+d(gy,hy)^2}$ 
(proposition \ref{prop- d_g(x) si x dans Min(h)}).

Pour le point \ref{item- d(x,Min(h))},  on raisonne par l'absurde.
Sinon il existe une suite $\xk\in\Es$ telle que 
$d_g(\xk)\tend\ell(g)$ et  $d(\xk,\Min(h)) \geq \eps$.
On note $\yk$ la projection  de
$\xk$ sur $\Min(h)$.
Alors $d(h\yk,g \yk) \tend 0$.
Fixons $y\in\Min(h)$.  Comme le commutateur $Z_G(h)$ de $h$
agit transitivement sur $\Min(h)$, 
on peut choisir $\zk\in Z_G(h)$
telle que $\zk\yk=y$.
Soit $\xpk=\zk\xk$ et $\gk=\zk g \zk^{-1}$. 
Alors $\gk y\tend hy$.
\begin{figure}[h]
\begin{center}
\input{unicite_vdt.pstex_t}
\end{center}
\end{figure}
Soit $\rk$ le rayon g{\'e}od{\'e}sique issu  de $y$ et passant par $\xpk$, 
et $\rpk=\gk\rk$.
Alors quitte {\`a} extraire $\rk \tend r$
avec $r$ orthogonal {\`a} $\Min(h)$ en $y$, et $\rpk\tend\rp$ 
avec $\rp$ orthogonal {\`a}
$\Min(h)$ en $h y$.
Pour $0<t<\eps$ fix{\'e}
et $k$ suffisamment grand, 
on a $\ell(\gk)\leq d(\rk(t),\rpk(t))\leq d(\xpk, \gk \xpk)$.
Or $\ell(\gk)=\ell(g)$ et $d(\xpk, \gk \xpk)=d_{g}(\xk)\tend\ell(g)$.
On a donc $d(r(t),r'(t))=\ell(g)=d(y,hy)$.
 Par cons{\'e}quent,
 la g{\'e}od{\'e}sique passant par $r(t)$ et $r'(t)$ est
parall{\`e}le {\`a} celle passant par $y$ et $hy$, 
qui est translat{\'e}e par $h$.
On a donc $r(t)\in\Min(h)$, 
ce qui est impossible car $r$ est
orthogonal {\`a} $\Min(h)$.
\end{preuve}

\subsection{C{\^o}nes asymptotiques}
\label{ss- prel conas}

On renvoie par exemple {\`a} \cite[2.4]{KlLe} 
pour les r{\'e}sultats rappel{\'e}s et non d{\'e}montr{\'e}s dans cette section.

\paragraph{Ultrafiltres.} Voir \cite[Ch. 1]{BouTG}.
Dans tout cet article, $\om$ d{\'e}signe un ultrafiltre sur $\NN$, plus
fin que le filtre de Fr{\'e}chet.
c'est-{\`a}-dire 
une mesure de probabilit{\'e} finiment additive d{\'e}finie sur toutes les
parties de $\NN$, {\`a} valeurs dans $\{0,1\}$, nulle sur les parties finies.
Une suite $(\xk)_{k\in\NN}$ dans un espace topologique  quelconque $E$
admet
$x\in E$ comme limite suivant l'ultrafiltre $\om$ ({\em \omlimite}) 
si pour tout
voisinage $V$ de $x$ dans $E$, on a $\xk\in V$  pour \omptt{} $k$. 
Cette limite est l'une des valeurs d'adh{\'e}rences de la suite
$(\xk)$, et elle est unique si $E$ est s{\'e}par{\'e}. 
On notera alors $\limom \xk =x$ ou $\xk\tendom x$.
Une suite contenue dans un compact admet toujours une \omlimite{}. Les
\omlimites{} des suites r{\'e}elles sont prises dans le compact $[\minfty,\pinfty]$.
On dit qu'une suite $\xk$ dans $\RR$ est \ommajoree{} (resp. \ombornee)
 si $\limom \xk <\pinfty$ (resp. si $\minfty<\limom \xk <\pinfty$).
On note $\eqom$ la relation d'{\'e}quivalence ``{\'e}galit{\'e} \ompp{}''.

\paragraph{Ultraproduits.}
\label{ss- ultraproduits}
On utilisera dans la section \ref{s- modele alg de Esom} le formalisme suivant.

Si $E$ est un ensemble, 
on notera 
$\et{E}$ son {\em ultraproduit},
c'est-{\`a}-dire l'ensemble  des suites d'{\'e}l{\'e}ments de $E$, 
 modulo $\eqom$.
On note $\classet{\xk}$ la classe de la suite
$(\xk)_{k\in\NN}$ modulo $\eqom$. 
L'ensemble $E$ se plonge canoniquement dans $\et{E}$ via les suites
constantes.  On identifiera $E$ et son image lorsque cela n'engendre
pas de confusion.
Si $E$ et $F$ sont deux ensembles et si $(\fk)_{k\in\NN}$ est une
suite d'applications de $E$ dans $F$, on note $\classet{\fk}$
l'application qui {\`a} $\classet{\xk}\in \et{E}$ associe
$\classet{\fk(\xk)}\in \et{F}$.

Si  $F$ est un espace topologique compact, on note 
$\limom {\fk}: \et{E}\fleche F$ l'application 
$\classet{\xk}\mapsto \limom{\fk(\xk)}$.
 C'est la compos{\'e}e de
$\limom : \et{F} \fleche F$ et de $\classet{\fk}$.

\paragraph{Ultralimites d'espaces m{\'e}triques point{\'e}s.}
Soit $(\Esk,\dk,\ok)_{k\in\NN}$ une suite d'espaces m{\'e}triques point{\'e}s.
On note
$$\Esetp=\{(\xk)_{k\in\NN}\in\prod_{k\in\NN} \Esk
\ |\  \dk(\ok,\xk) \mbox{ est born{\'e}e}\}\; ,$$  
l'ensemble  des suites {\em \ombornees}. 
Il est muni de la pseudo-distance 
$\dom$ d{\'e}finie par  $\dom((\xk),(\yk))=\limom \dk(\xk,\yk)$. 
Le quotient de $\Esetp$ par la relation
``\^etre {\`a} pseudo-distance $\dom$ nulle'' est un espace m{\'e}trique
$(\Esom,\dom)$, point{\'e} en $\oom=[\ok]$, appel{\'e}
{\em ultralimite} de la suite
$(\Esk,\dk,\ok)_{k\in\NN}$.
Pour $(\xk)\in\Esetp$, on note $\classom{\xk}$ le
point de $\Esom$ {\em associ{\'e}}. 
Si $A_k\subset\Esk$ on note
$\classom{A_k}=\{\classom{\xk};\ (\xk)\in\Esetp \mbox{ et } \forall k, \xk\in
A_k\}$.

\begin{prop}
L'espace m{\'e}trique $\Esom$ est complet. \cite[Lemma 2.4.2]{KlLe}\qed
\end{prop}

\paragraph{C{\^o}nes asymptotiques d'un espace m{\'e}trique.}
On se fixe d{\'e}sormais un espace m{\'e}trique $\Es$, une suite 
$\ok\in\Es$, et une suite $\lak\in\RR_+^*$, avec $\lak\tend\pinfty$. 
On se donne {\'e}galement un groupe $G$ agissant
par isom{\'e}tries sur $\Es$.
Le {\em c{\^o}ne asymptotique} de $\Es$, suivant $\om$, par
rapport {\`a} la {\em suite de  points d'observation} $(\ok)_{k\in \NN}$ et
{\`a} la {\em suite de scalaires} $(\lak)_k$ est par d{\'e}finition l'ultralimite
$(\Esom,\dom,\oom)$ de la suite d'espaces m{\'e}triques point{\'e}s $(\Esk,
\dk=\lslak d, \ok)_{k\in\NN}$.
Remarquons que si $\Es$ est euclidien, alors $\Esom$ est isom{\'e}trique {\`a}
$\Es$ (car il existe des applications de $\Es$ dans lui-m{\^e}me envoyant un
point donn{\'e} sur $\ok$ et $d$ sur  $\lslak d$). 
Soit $\Esp$ un autre espace m{\'e}trique, et une suite
$\opk\in\Esp$. Notons $\Espom$ le c{\^o}ne asymptotique de $(\Esp,(\opk),(\lak)$.
Une suite d'applications
isom{\'e}triques  $\fk:\Esp\fleche \Es$  est dite \ombornee{}
si $\fk(\opk)$ est \ombornee{} dans $\Es$. 
Alors on peut d{\'e}finir son {\em \ulimite} $\fom:\Espom\fleche\Esom$,
not{\'e}e aussi $\ulim{\fk}$, par $\fom(\classom{\xk})=\classom{\fk(\xk)}$. 
Elle est isom{\'e}trique.
En particulier, si $\Esp=\RR$, une suite $\rk:\RR\fleche\Es$ de
g{\'e}od{\'e}siques de $\Es$ est \ombornee{} si la suite $\lslak
d(\ok,\rk(0))$ est \omborne. Son \ulimite{} s'identifie {\`a} 
$\rom:t\mapsto\rk(\lak t)]$ et c'est une g{\'e}od{\'e}sique de $\Esom$.
On note $\Getp$ l'ensemble des suites $(\gk)_{k\in \NN}$ d'{\'e}l{\'e}ments de
$G$ telles que $\dk(\ok,\gk\ok)$ est \omborne. Pour $(\gk)\in\Getp$, on
note $[\gk]:[\xk]\mapsto[\gk \xk]$ l'isom{\'e}trie {\em associ{\'e}e} de
$\Esom$. On note $\Gom$ le sous-groupe de $\Isom(\Esom)$ 
form{\'e} par ces isom{\'e}tries.

\paragraph{C{\^o}nes asymptotiques d'espaces $\CAT(0)$.}
On suppose d{\'e}sormais que $\Es$ est  $\CAT(0)$.

\begin{prop} 
\label{prop- passage des angles au conas}
\begin{enumerate}

\item 
\label{item- Esom CAT(0)} 
Le c{\^o}ne asymptotique $\Esom$ est $\CAT(0)$. 

\item 
\label{i- geods de Esom}
Les g{\'e}od{\'e}siques de $\Esom$ sont les \omlimites{} des g{\'e}o\-d{\'e}\-si\-ques de $\Es$.

\item
\label{item- continuite asymp angle comp}
\label{item- l'angle augmente}
Pour tous $\xom=[\xk]$, $\yom=[\yk]$ et $\zom=[\zk]$ dans $\Esom$, on a
$$\angletilde_\xom(\yom,\zom) = \limom \angletilde_{\xk}(\yk,\zk)
\ \ \mbox{ et }\ \  \angle_\xom(\yom,\zom) \geq \limom \angle_{\xk}(\yk,\zk).$$

\item 
\label{item- om-limite projections orthogonales}
Soit $\rk$ une
suite \ombornee{} de g{\'e}od{\'e}siques et $\rom=[\rk]$.  
La projection orthogonale sur
$\rom$ est la \omlimite{} de la suite des projections orthogonales
 sur $\rk$. 

\end{enumerate}
\end{prop}

\begin{preuve} 
Le point \ref{item- Esom CAT(0)} et \ref{i- geods de Esom} 
d{\'e}coulent du lemme 2.4.4 de \cite{KlLe}. 

La premi{\`e}re assertion du point 
\ref{item- continuite asymp angle comp} 
est facile {\`a} voir en regardant les triangles de
comparaison dans le plan (car les angles sont inchang{\'e}s par homoth{\'e}ties). 
La deuxi{\`e}me en d{\'e}coule. En effet 
soient $\ypom=\classom{\ypk}$ et
$\zpom=\classom{\zpk}$ avec $\ypk\in]\xk, \yk]$ et
$\zpk\in]\xk,\zk]$.
On a 
$\angletilde_{\xpk}(\ypk,\zpk)\geq \angle_{\xpk}(\ypk,\zpk)$ et
$\angle_{\xk}(\yk,\zk)=\angle_{\xk}(\yk,\zk)$ pour tout $k$. 
Donc en passant {\`a} la limite suivant $\om$ on a
$\angletilde_\xpom(\ypom,\zpom)\geq \limom \angle_{\xk}(\yk,\zk)$.
On conclut en passant {\`a} la limite pour $\ypom \rightarrow \xom$ et
$\zpom \rightarrow \xom$.  

Le point \ref{item- om-limite projections orthogonales}
d{\'e}coule du fait que les angles augmentent 
(point \ref{item- continuite asymp angle comp}).
\end{preuve}

\paragraph{Bord {\`a} l'infini.} 
Il est clair que les \omlimites{} de rayons asymptotes restent asymptotes.
\`A toute suite $(\xik)_k$ dans $\bordinf\Es$ on peut donc associer un
point $\xiom=[\xik]$ dans $\bordinf\Esom$, d{\'e}fini de la mani{\`e}re suivante. 
Soit $\rk$ une suite \ombornee{} de rayons de $\Es$ tels que
$\xik=\rk(\infty)$ et $\rom=[\rk]$. Alors $\xiom=\rom(\infty)$.
Tous les points de $\bordinf\Esom$ sont de cette forme.

\paragraph{Longueur de translation.}
Le lemme suivant 
se v{\'e}rifie ais{\'e}ment.

\begin{lemm} \label{lemm- Min(gk) reste a distance ombornee}
 Soit $\gk$ une suite $\om$-born{\'e}e dans $\Isom(\Es)$ et $\gom=[\gk]$. 
On a 
 \begin{enumerate}
 \item $\limom \lslak d_{\gk} =d_\gom$
 \item $\limom \lslak\lgk\leq \lgom$.
 \item $[\Min(\gk)]\subset\Min(\gom)$.
\item  Si $\lslak d(\ok, \Min(\gk))$ est \omborne{}, 
alors $\Min(\gom)\neq\emptyset$ et $\limom \lslak\lgk = \lgom$
\qed
 \end{enumerate}
\end{lemm}

\paragraph{C{\^o}nes asymptotiques d'espaces sym{\'e}triques et d'immeubles.}
On suppose dor{\'e}navant que  $\Es$ est ou bien un espace
sym{\'e}trique, ou bien un immeuble affine, 
et que $G$ est un sous-groupe ferm{\'e} du groupe des automorphismes de $\Es$.
On se fixe un plat maximal $\Aa$ de $\Es$, de groupe de Weyl $\Wvect$ et
une chambre de Weyl $\Cb$ de $\Aa$. 
Une suite $\fk:\Aa\fleche\Es$ de plats maximaux marqu{\'e}s 
est \ombornee{}
si la suite $\fk(0)$ l'est, et son \omlimite{} est alors l'application
isom{\'e}trique $[\fk]:\Aa\fleche\Esom,\ \af\mapsto [\fk(\lak \af)]$.
On note $\Aom$ l'ensemble des \omlimites{} des suites \ombornees{} de
plats maximaux marqu{\'e}s de $\Es$. 

\begin{theo}[Kleiner-Leeb \cite{KlLe}]
\label{theo- cone as = immeuble} 
Le c{\^o}ne asymptotique $\Esom$, mu\-ni du sys\-t{\`e}\-me d'appartements
marqu{\'e}s $\Aom$, est un immeuble affine complet, de type
$(\Aa,\Wvect)$.
Le groupe $\Gom$ agit sur $\Esom$, par automorphismes.
Si  $G$ agit cocompactement sur $\Es$, alors $\Gom$
agit transitivement sur $\Esom$. 
Si $G$ agit transitivement sur les plats maximaux marqu{\'e}s de $\Es$,
alors $\Gom$ agit transitivement sur $\Aom$.
\label{Gom transitif sur les appartements marques de Esom}%
\qed 
\end{theo}

On termine par deux propri{\'e}t{\'e}s du type des segments dont on aura
besoin par la suite, et qui se v{\'e}rifient ais{\'e}ment.

\begin{prop}
\label{prop- type asympt continu}
Pour tous  $\xom=[\xk]$ et $\yom=[\yk]$ de $\Esom$, on a dans $\Cb$
\[\Cd(\xom,\yom)=\limom\lslak\Cd(\xk, \yk)\; . \]\qed  \end{prop}

\begin{prop}  
\label{prop- suite de geodesiques de type constant}
\label{prop- Gom geod type constant}
 Soit $\rom$ une g{\'e}od{\'e}sique de $\Esom$. 
Il existe une suite $(\rk)$ \ombornee{} de g{\'e}od{\'e}siques de $\Es$
de type constant $\btype (\rom)$
telle que $[\rk]=\rom$.
Si  $G$ agit cocompactement sur $\Es$ et transitivement sur
les plats maximaux marqu{\'e}s de $\Es$, alors
$\Gom$ agit transitivement sur les g{\'e}od{\'e}siques de $\Esom$ de
type donn{\'e}.
\qed\end{prop}

\section{C{\^o}nes asymptotiques d'espaces
  sym{\'e}triques  : Mod{\`e}le alg{\'e}brique} 
\label{s- modele alg de Esom}

 Soit $\KK$  un corps valu{\'e}
ou bien {\'e}gal {\`a} $\RR$ ou $\CC$ ou bien ultram{\'e}trique (la valuation
n'est alors pas n{\'e}cessairement discr{\`e}te). On note $\abs{\ }$ la valeur absolue de $\KK$.
On note $\Es(\KK)$ ou $\Es$ l'espace associ{\'e} {\`a} $G=\SLn(\KK)$ (voir
section \ref{ss- immeuble d'un groupe reductif sur un corps non
  arch}).

Cette section est consacr{\'e}e {\`a} montrer que tout c{\^o}ne asymptotique de
$\Es(\KK)$ s'identifie {\`a} $\Es(\KKom)$ (th{\'e}or{\`e}me \ref{theo- modele
  alg Esom}), pour un corps valu{\'e} non archim{\'e}dien ad hoc $\KKom$ qu'on
construit pr{\'e}alablement (dont la valeur absolue prend toutes les
valeurs r{\'e}elles).

La preuve repose sur l'utilisation du mod{\`e}le des ``bonnes
normes'' de volume $1$ sur $\KK^n$ pour $\Es(\KK)$, qu'on expose au
pr{\'e}alable en section \ref{ss- normes}.

\subsection{\Pseudovas{} et \pseudonormes{}}
On commence par clarifier quelques notions de base qu'on va utiliser
plusieurs fois par la suite.
Soit $\FF$ un corps et $E$ un espace vectoriel sur $\FF$.

\begin{defi}[\Pseudova]
On appellera {\em \pseudova} sur $\FF$  une application
$\abs{\ }:\FF\fleche[0,\pinfty]$ telle que 
pour tous $a,b$ dans $\FF$
on a
\begin{enumerate}
\item $\abs{0}=0$
\item 
\label{item- abs(ab)} 
$\abs{ab} = \abs{a}\abs{b}$ (lorsque $\abs{a},\abs{b}<\pinfty$)
\item \label{item- abs(a+b)}
 $\abs{a+b} \leq \abs{a}+\abs{b}$
\end{enumerate}

On dira que $\abs{\ }$  est {\em ultram{\'e}trique} si elle
v{\'e}rifie en outre 
$\abs{a+b}\leq \max(\abs{a},\abs{b})$. 

 Si le corps $\FF$ est muni d'un ordre total $\leq$, on dira que
$\abs{\ }$ est {\em compatible} avec l'ordre si 
pour tous $a,b\in\FF$ tels que
$0\leq b\leq a$, on a  $\abs{b} \leq \abs{a}$.

Une \pseudova{} $\abs{\ }$ est une {\em valeur
absolue} si elle v{\'e}rifie $\abs{a}<\pinfty$ et
$\abs{a}=0\Rightarrow a=0$. 
On dira qu'elle est {\em triviale} si $\abs{a}=1$ pour tout 
$a\neq 0$.

\end{defi}

\begin{prop}
\label{prop- corps value ordonne}
On suppose que $(\FF,\leq)$ est un corps totalement ordonn{\'e} et que
$\abs{\ }$ est une \pseudova{} sur $\FF$ compatible avec l'ordre.

\begin{enumerate}
\item %
Si $\abs{\ }$ est ultram{\'e}trique, alors pour tous $a,b\geq 0$ dans $\FF$, on a 
\[\abs{a+b} = \max(\abs{a},\abs{b})\]

\item Si $\abs{\ }$ est triviale sur le sous-corps premier de $\FF$,
alors elle est ultram{\'e}trique. \qed
\end{enumerate}
\end{prop}

\begin{preuve}
La premi{\`e}re assertion est claire.
Soient $a,b\geq 0$ dans $\FF$.  
On peut supposer que $b\leq a$. On a alors
  $0\leq a+b\leq 2a$ donc 
$\abs{a+b}\leq  \abs{2a}=\abs{2}\abs{a}=\abs{a}$.
\end{preuve}

On suppose d{\'e}sormais que $\FF$ est muni d'une \pseudova{} $\abs{\ }$.
On d{\'e}\-mon\-tre sans difficult{\'e}s les r{\'e}sultats suivants.

\begin{prop}[Corps valu{\'e} associ{\'e}]
\label{prop- corps value associe a une pseudova}
 Notons
  \(\FFp=\{a\in\FF\ |\ \abs{a}<\pinfty\}\) et
  \(\FFo=\{a\in\FF\ |\ \abs{a}=0\}\). 
Alors $\FFp$ est un  sous-anneau de $\FF$ et $\FFo$ est un id{\'e}al
  maximal de $\FFp$.
La \pseudova{} $\abs{\ }$ passe au quotient en une valeur absolue
sur le corps quotient $\FF_{/\abs{\ }}=\FFp/\FFo$.

Si le corps $\FF$ est muni d'un ordre total $\leq$ compatible avec $\abs{\ }$,
alors $\leq$ passe au quotient en un ordre total sur le corps $\FF_{/\abs{\ }}$
compatible avec $\abs{\ }$.
\qed
\end{prop}

\begin{defi}[\Pseudonorme]
On appellera {\em \pseudonorme} sur $E$ (compatible avec $\abs{\ }$)  
une application $\norm:E \fleche [0,\pinfty]$ 
telle que pour tous $v,v'\in E$ et $a\in\FF$
on a
\begin{enumerate}
\item $\norm(0)=0$
\item 
$\norm(av) = \abs{a}\norm(v)$ (lorsque cela a un sens)
\item 
$\norm(v+v') \leq \norm(v)+\norm(v')$
\end{enumerate}

On dira que $\norm$  est {\em ultram{\'e}trique} si on a
en outre 
$\norm(a+b) \leq \max(\norm(a),\norm(b))$. 
On dira que $\norm$  est une {\em norme} 
si $\abs{\ }$ est une valeur absolue et 
si $\norm$ v{\'e}rifie  en outre $\norm(v)<\pinfty$ et
$\norm(v)=0\Rightarrow v=0$. 
\end{defi}

Soit  $\norm$ une \pseudonorme{} sur $E$ (compatible avec $\abs{\ }$).

\begin{prop}[Espace vectoriel norm{\'e} associ{\'e}]
\label{prop- evn associe a une pseudonorme}
On a deux sous-$\FFp$ modules
  \(\Ep=\{v\in E\ |\ \norm(v)<\pinfty\}\) et
  \(\Eo=\{v\in E\ |\ \norm(v)=0\}\). 
La structure passe au quotient 
en une structure de $\FF_{/\abs{\ }}$-espace vectoriel 
sur $E_{/\norm}=\Ep/\Eo$,
et $\norm$ passe au quotient en une norme
sur $E_{/\norm}$ compatible avec $\abs{\ }$.
\qed
\end{prop}

On utilisera {\'e}galement la propri{\'e}t{\'e}
analogue suivante.
 \begin{prop}
\label{prop- algn  associe a une  pseudonorme}

Soit $A$ une alg{\`e}bre norm{\'e}e sur $\FF$ et $N$ une \pseudonorme{} sur
$A$ sous-multiplicative.
Notons
  \(\Ap=\{u\in A\ |\ N(u)<\pinfty\}\) et
  \(\Ao=\{v\in A\ |\ N(u)=0\}\). 
La structure passe au quotient 
en une structure de $\FF_{/\abs{\ }}$-alg{\`e}bre
sur $\A_{/N}=\Ap/\Ao$,
et $\norm$ passe au quotient en une norme
sous-multiplicative sur $A_{/N}$ compatible avec $\abs{\ }$.
\qed
 \end{prop}

\begin{defi}[\Pseudonorme{} d'endomorphismes]
On d{\'e}finit la  {\em \pseudonorme{}}
$\Nnorm$ sur $\End(E)$
{\em associ{\'e}e} {\`a}  $\norm$   par
\[\Nnorm(u)=\inf\left\{M\in [0,\pinfty]\ 
;\ \forall v\in E, \norm(u(v))\leq M\norm(v)\right\}\]
On v{\'e}rifie que 
c'est une \pseudonorme{} sous-multiplicative sur $\End(E)$.
\end{defi}

\subsection{Espace $\Es=\BNl(V)$ des bonnes normes de volume $1$}
\label{ss- normes}
On expose ici la construction de l'espace $\Es$ 
associ{\'e} {\`a} $\SLn(\KK)$ comme espace de normes sur
$\KK^n$ (voir aussi \cite{ParImm}).

 On note $V=\KK^n$ et $\eo=(\eoln)$ sa base canonique.

Dans le cas o{\`u} $\KK=\RR$ ou $\CC$, 
on dira qu'une norme $\norm$ sur $V$ est une {\em bonne} norme si
elle est {\em euclidienne}, c'est-{\`a}-dire si elle est de la forme $v\mapsto
\sqrt{q(v)}$, o{\`u} $q$ est 
 une forme quadratique si $\KK=\RR$, hermitienne si $\KK=\CC$, d{\'e}finie positive 
sur $V$.
On dira qu'une base $\e=(\eln)$ de $V$ est {\em adapt{\'e}e} {\`a} $\norm$ si
elle est {\em orthogonale} pour $\norm$, c'est-{\`a}-dire si on a
$\norm(\sum_{i=1}^n a_ie_i)=\sqrt{\sum_{i=1}^n \norm(e_i)^2\abs{a_i}^2}$ 
pour tous $\aln$ dans $\KK$.  
La bonne norme {\em canonique} $\normo$ de $V$ est la norme
provenant du produit scalaire canonique.
Dans le cas o{\`u} $\KK$ est ultram{\'e}trique, 
Une base $\e$ de $V$ est dite {\em adapt{\'e}e} {\`a} 
une norme ultram{\'e}trique $\norm$ si
on a
$\norm(\sum_{i=1}^n a_ie_i)=\max\limits_\ilan \norm(e_i)\abs{a_i}$ 
pour tous $\aln$ dans $\KK$.
On dira que $\norm$ est
{\em adaptable} si elle admet une base  adapt{\'e}e.
(ce qui est toujours le cas si $\KK$ est un corps local \cite{GoIw}).
Une {\em bonne} norme est une norme ultram{\'e}trique et adaptable.
La bonne norme {\em canonique} est
$\normo:(\aln)\mapsto\max\limits_\ilan \abs{a_i}$.

On dit qu'une bonne norme $\norm$ est de {\em volume} $1$ si on a
$\prod_{i=1}^n \norm(e_i) =1$ pour toute base $(\eln)$ de $V$
adapt{\'e}e {\`a} $\norm$, de d{\'e}terminant $1$.
On note $\BNl(V)$ l'espace des bonnes normes de volume
$1$ sur $V$.

Le groupe $G=\SLn(\KK)$ agit naturellement sur $\BNl(V)$  par
$g.\norm=\norm\circ g^{-1}$ pour $g\in \SLn(\KK)$ et $\norm\in\BNl(V)$.
Le stabilisateur de $\normo$ est $\SO(n)$ si $\KK=\RR$, le groupe
$\SU(n)$ si $\KK=\CC$, et $\SLn(\mathcal{O})$ si $\KK$ est
ultram{\'e}trique et $\mathcal{O}=\{a\in\KK,\abs{a}\leq 1\}$ est son anneau des entiers.
Dans le cas o{\`u} $\KK=\RR$ ou $\CC$,  l'espace $\BNl(V)$ s'identifie {\`a} l'espace sym{\'e}trique $\Es$
associ{\'e} {\`a} $\SLn(\KK)$ (voir \ref{ss- espace sym associe a un groupe reductif reel}).
Dans le cas o{\`u} $\KK$ est ultram{\'e}trique,
l'espace $\BNl(V)$
est alors un immeuble affine (voir par exemple \cite{ParImm}), qui
s'identifie {\`a} l'immeuble de
Bruhat-Tits $\Es$ de $\SLn(\KK)$, voir les th{\'e}or{\`e}mes 2.11 et 2.11bis
(dans l'appendice) de \cite{BrTi_schemas}.

\medskip

Rappelons maintenant quelques propri{\'e}t{\'e}s  dont on aura
besoin par la suite.

\paragraph{Plats maximaux marqu{\'e}s.}
Le plat maximal mod{\`e}le  de $\BNl(V)$ est identifi{\'e} {\`a}
 \[\Aa=\{(\afln) \in \RR^n,\ \sum \afi = 0\}\]
muni de la m{\'e}trique euclidienne canonique de $\RR^n$.
Soit $\e$ une base de $V$ de d{\'e}\-ter\-mi\-nant $1$.
Le plat maximal marqu{\'e} $f_\e:\Aa \fleche \BNl(V)$
de $\BNl(V)$ {\em associ{\'e}} {\`a} $\e$ est d{\'e}fini par :
$\norm=f(\af)$ est adapt{\'e}e {\`a} $\e$ et $\norm(e_i)=e^{-\afi}$ pour
$\ilan$.
Tous les plats maximaux marqu{\'e}s de $\BNl(V)$ sont
de cette forme.
Pour tout $g\in \SLn(\KK)$, on a $g\circ f_\e=f_{g\e}$.  

\paragraph{Les distances $d$ et $\dinf$.} 
Si $\norm,\normp\in\BNl(V)$
et $\e$ est une base adapt{\'e}e commune, 
alors 
\begin{equation}  \label{eq- d} 
d(\norm,\normp)
=\sqrt{\sum_{i=1}^n \abs{\log {\normp\over\norm}(e_i)}^2}\; .
\end{equation} 
Un autre distance naturelle plus simple sur $\BNl(V)$ est
\[\dinf(\norm,\normp)=\sup\limits_{v\neq 0}\abs{\log
  {\normp\over\norm}(v)} \; .\]
Les distances $d$ et $\dinf$  sont {\'e}quivalentes, car 
\cite[prop. 3.12]{ParImm}
\begin{equation} \label{eq- d et dinf}
\dinf \leq d\leq \sqrt{n}\ \dinf \mbox{\ \ sur } \BNl(V).
\end{equation}

\paragraph{Estimations.}
Soit $\Neo=N_{\normo}$ la norme sur $\End(V)$ associ{\'e}e {\`a} $\normo$. 
Pour tout $g\in \SLn(\KK)$, on a 
$\sup\limits_{v\neq 0}{g.\normo\over\normo}(v) = \Neo(g^{-1})$
et donc
\begin{equation} \label{eq- d neuo et N(g)}
  \begin{array}{ccl}
      \dinf(\normo,g.\normo) &=&
      \max \left\{ \log \Neo(g),\ \log \Neo(g^{-1})\right\}\\
      &\leq& (n-1)\;\log \Neo(g)
  \end{array}
\end{equation}
Le lemme suivant sera utilis{\'e} en section \ref{ss- conas et imm de BT}. 

\begin{lemm} \label{lemm- d neu, g.neu et N(g-I)}
Soit $\norm\in \BNl(V)$
et $g\in \SLn(\KK)$. On a 
\[{\dinf(\norm,g.\norm)}\leq \log \left( 1+ e^{2\dinf(\normo, \norm)} \max \left\{
\Neo(g-\id),\ \Neo(g^{-1}-\id) \right\} \right)\;.\]
\end{lemm}

\begin{preuve}
Pour tout $v\in V$ non nul, on a
$g.\norm(v) \leq \norm(v)+\norm((g^{-1}-\id)v)$, et  
en posant $h=g^{-1}-\id$ on a
\[\norm(hv) \leq e^{\dinf(\normo, \norm)}\normo(hv)
\leq e^{\dinf(\normo, \norm)}\Neo(h)\normo(v)
\leq e^{2\dinf(\normo, \norm)}\Neo(h)\norm(v)\]
donc $\sup\limits_{v\neq 0}{g.\norm\over\norm}(v) \leq
1+e^{2\dinf(\normo, \norm)} \Neo(g^{-1}-\id)$.
On conclut en {\'e}changeant $\norm$ et $g.\norm$.
\end{preuve}

\medskip

{\bf On se fixe d{\'e}sormais 
un ultrafiltre $\om$ sur $\NN$ plus fin que
le filtre de Fr{\'e}chet (voir \ref{ss- prel conas})
et
une suite $\la=(\lak)_{k\in\NN}$ dans $[1,\pinfty[$ telle que
      $\lak\tend\infty$.
}

\subsection{Le c{\^o}ne asymptotique $\KKom$ du corps valu{\'e} $\KK$}
\label{ss- corps ultralimite}

\begin{prop}
 \label{prop- absom}\ 
On munit l'ensemble
\[\KKom=
\{(\ak)\in\KK^\NN\ |\ \abs{\ak}^\lslakh \mbox{ est
  born{\'e}}\}_{/\limom\abs{\ak-\bk}^\lslakh= 0}\] de l'addition et
de la multiplication terme {\`a} terme. 
On plonge $\KK$ dans $\KKom$ via les suites constantes.
Pour $\aom=\classom{\ak}$ dans $\KKom$, on note
 \[\absom{\aom}=\limom \abs{\ak}^\lslakh\]
Alors $\KKom$ est un corps et $\absom{\ }$ est une valeur absolue
ultram{\'e}trique sur $\KKom$, triviale sur $\KK$. 
Ce corps valu{\'e}
sera appel{\'e} le {\em c{\^o}ne asymptotique} de $(\KK,\abs{\ })$ par rapport
{\`a} la suite de scalaires $(\lak)$.

\medskip

Dans le cas o{\`u} $\KK=\RR$, on peut d{\'e}finir une relation
d'ordre total sur $\KKom$ par
\[\aom\leqom \bom \Leftrightarrow \ak\leq \bk \mbox{ \ompp}\]
Alors $\absom{\ }$ est compatible avec l'ordre.
\end{prop}

\begin{rema}
\label{rema- KKom}
  On notera que, si $\dk$ d{\'e}signe la distance sur $\KK$ induite par la
valeur absolue $\abs{\ }^\lslak$ 
 et $\dom$ d{\'e}signe la distance sur $\KKom$ induite par la
valeur absolue $\absom{\ }$, alors l'espace m{\'e}trique point{\'e}
$(\KKom,\dom,0)$ est l'ultralimite des espaces m{\'e}triques point{\'e}s
$(\KK,\dk,0)$ (en particulier $\KKom$ est complet).

\end{rema}

\begin{preuve} 
On utilise la description alternative de $\KKom$ via l'ultraproduit
$\KKet$ de $\KK$. On munit $\KKet$ des op{\'e}rations termes {\`a} terme, ce
qui en fait un corps. Pour $\aet=\classet{\ak}\in \KKet$ on d{\'e}finit
 $\abslaet{\aet}=\limom \abs{\ak}^\lslakh$
dans $[0,\pinfty]$.  
La relation $\leqet$ sur $\RRet$ d{\'e}finie par
$$\classet{a_k}\leqet \classet{b_k} 
\mbox{ si et seulement si } 
a_k\leq b_k\mbox{ \ompp{}}$$ 
fait de $\RRet$ un corps totalement ordonn{\'e}.
D'apr{\`e}s la proposition \ref{prop- corps value associe a une
  pseudova},  et la proposition \ref{prop- corps value ordonne}, 
il suffit de prouver que $\abslaet{\ }$ est une \pseudova{} sur
$\KKet$, triviale sur $\KK$, compatible avec l'ordre si $\KK=\RR$. 
On le v{\'e}rifie sans difficult{\'e}.
On a alors clairement que $(\KKom, \absom{\ })$ est le
corps valu{\'e} canoniquement associ{\'e} au corps pseudo-valu{\'e} $(\KKet,
\abslaet{\ })$.
\end{preuve}

\begin{rema}
1) Pour tout $\aom=\classom{a_k}\geq 0$ dans $\RRom$, il existe un
unique {\'e}l{\'e}ment positif de $\RRom$ de carr{\'e} $\aom$, qui est
$\sqrt{\aom}=\classom{\sqrt{a_k}}$.

2) Pour tous $\alnom$ dans $\RRom$, on a 
\begin{equation}
 \label{eq- supp} 
\absom{\sqrt{\psum_{i=1}^n (\aiom)^2}\;}  
=\max\limits_\ilan\absom{\aiom}
\end{equation}
(cela d{\'e}coule de
$\absom{\sqrt{\aom}\;}=\sqrt{\absom{\aom}}$ et
$\absom{\aom+\bom}=\max(\aom,\bom)$ pour $\aom,\bom\geqom 0$).
\end{rema}

\subsection{C{\^o}ne asymptotique de l'espace vectoriel norm{\'e} $V$}

L'espace vectoriel $V=\KK^n$ est muni de sa norme
canonique $\normo$ (voir section \ref{ss- normes}).

\begin{rema}
\label{rema- coeffs}
  Si  $\vk=(a^1_k,\ldots,a^n_k)$  dans $V=\KK^n$, 
on v{\'e}rifie sans difficult{\'e} (par (\ref{eq- supp}) dans le cas
archim{\'e}dien) que
\begin{equation}
  \limom\normo(\vk)^\lslakh=\max_i\limom\abs{\aik}^\lslakh
 \end{equation}
\end{rema}

\begin{prop}
 \label{prop- conas de V}\ 
On note
\[\Vom=
\{(\vk)\in V^\NN\ |\ \normo(\vk)^\lslakh \mbox{ est
  born{\'e}}\}_{/\limom\normo(\vk - \vpk)^\lslakh= 0}\] Alors l'addition et
la multiplication terme {\`a} terme d{\'e}finissent une structure de
$\KKom$-espace vectoriel sur $\Vom$.
Pour $\vom=\classom{\vk}$ dans $\Vom$, on note
 \[\normoom(\vom)=\limom \normo(\vk)^\lslakh\]
Alors  $\normoom$ est une norme ultram{\'e}trique sur $\Vom$, compatible
 avec $\absom{\ }$.
Cet espace vectoriel norm{\'e} 
sera appel{\'e} le {\em c{\^o}ne asymptotique} de $(V,\normo)$ par rapport
{\`a} la suite de scalaires $(\lak)$.

\end{prop}

\begin{preuve}
  On utilise 
l'ultraproduit $\Vet$ de $V$, muni des op{\'e}rations terme {\`a}
terme, ce qui en fait un espace vectoriel sur $\KKet$.
Pour $\vet=\classet{\vk}\in \Vet$ on d{\'e}finit
 $\normoet(\vet)=\limom \normo(\vk)^\lslakh$ dans $[0,\pinfty]$.  
D'apr{\`e}s la proposition \ref{prop- evn  associe a une
  pseudonorme}
il suffit de prouver que $\normoet$ est une \pseudonorme{} sur
$\Vet$ compatible avec $\abslaet{\ }$.
On le v{\'e}rifie sans difficult{\'e}.
On a alors clairement que $(\Vom, \normoom)$ est l'espace vectoriel norm{\'e} 
canoniquement associ{\'e} {\`a} $(\Vet,\normoet)$.
Il est clair que $\normoom$ est ultram{\'e}trique par la remarque
\ref{rema- coeffs}.
\end{preuve}
On note $\vom=\classVom{\vk}$ le vecteur de $\Vom$ d{\'e}fini par
$\vet=\classet{\vk}\in\Vetp$.

\paragraph{Identification avec $(\KKom)^n$.}

La proposition suivante d{\'e}coule de la remarque \ref{rema- coeffs}

\begin{prop}
Pour $\ilan$, soit $\eoiom=\classVom{\eoi}$.  Alors
$\eoom=(\eoiom)_\ilan$ est une base de $\Vom$.  Elle
induit un isomorphisme d'espace vectoriels norm{\'e}s entre $\Vom=\omq{(\KK^n)}$
muni de $\normoom$ et $(\KKom)^n$ muni de la bonne norme
canonique. Plus pr{\'e}cis{\'e}ment,
pour tout $\vom$ de $\Vom$ de coordonn{\'e}es $(\alnom)$ dans la base
$\eoom$, on a
 \[\normoom(\vom)=\max\limits_\ilan \abslaet{\aiom}\] 
\end{prop}

On identifiera dor{\'e}navant ces deux espaces, et on les notera $\KKom^n$.

\subsection{C{\^o}ne asymptotique de $\End(V)$ et de $\SLn(\KK)$}
L'alg{\`e}bre $\End(V)$
est munie de la
norme $\Neo$ associ{\'e}e {\`a} $\normo$.

\begin{prop}[C{\^o}ne asymptotique de $\End(V)$]
 \label{prop- conas de End(V)}\ 
On note
\[\omq{\End(V)}=
\{(\uk)\in \End(V)^\NN\ |\ \Neo(\uk)^\lslakh \mbox{ est
  born{\'e}}\}_{/\limom\Neo(\uk-\upk)^\lslakh= 0}\]
Les op{\'e}rations terme {\`a} terme d{\'e}finissent une structure
de $\KKom$-alg{\`e}bre sur $\omq{\End(V)}$. 

Pour $\uom=\classom{\uk}$ dans $\omq{\End(V)}$, on note
 \[\Neoom(\uom)=\limom \Neo(\uk)^\lslakh\]
Alors  $\Neoom$ est une norme ultram{\'e}trique sous-multiplicative sur 
$\omq{\End(V)}$, compatible
 avec $\absom{\ }$.
Cette alg{\`e}bre  norm{\'e}e 
sera appel{\'e} le {\em c{\^o}ne asymptotique} de $(\End(V),\Neo)$ par rapport
{\`a} la suite de scalaires $(\lak)$.

\end{prop}

\begin{preuve}
De m{\^e}me qu'avant, 
il suffit de prouver que 
 $\Neoet(\uet)=\limom \Neo(\uk)^\lslakh$
d{\'e}finit une \pseudonorme{}
sous-multiplicative sur
$\et{\End(V)}$ compatible avec $\abslaet{\ }$ (voir proposition 
\ref{prop- algn  associe a une  pseudonorme}).
On le v{\'e}rifie sans difficult{\'e}.
\end{preuve}

\begin{defi}[C{\^o}ne asymptotique $\omq{\SLn(\KK)}$de $\SLn(\KK)$]
  Si $U\subset \End(V)$
on note $\omq{U}=\{\classEndVom{\uk}\ |\ \forall k\in\NN,\uk\in U\}$. 
Alors $\omq{\GL(V)}$ est un groupe pour la loi induite par la
composition, et $\omq{\SLn(\KK)}$ est l'un de ses sous-groupes.
\end{defi}

\paragraph{Identification de $\omq{\End(V)}$ avec $\End(\Vom)$.}

On rappelle que $\Vom=(\KKom)^n$ est muni de la bonne norme canonique
$\normoom$.  L'alg{\`e}bre $\End(\Vom)\simeq \Mn(\KKom)$ est munie de la
norme $\Neoomm=N_{\normoom}$, qui v{\'e}rifie
 \[\Neoomm((\aijom)_{ij})=\max_{ij} \absom{\aijom}\]

\begin{prop}
Soit $\uom=\classom{\uk}\in \omEndV$. 
On peut lui associer $\uomm\in\End(\Vom)$ d{\'e}fini par
$$\uomm:\ \classom{\vk} \mapsto \classom{\uk(\vk)}\;.$$ 
On note $(\aijom)_{ij}$ la matrice de $\uomm$ dans la base canonique, 
et $(\aijk)_{ij}$  la matrice de $\uk$ dans la base canonique, pour
tout $k$.
On a
\begin{enumerate}
\item 
$\Neoomm(\uomm) =  \Neoom(\uom)$.

\item $\det(\uomm) = \classKom{\det \uk}$. 

\item  $\aijom=\classKom{\aijk}$ pour
tous $i,j$.

\end{enumerate}
\end{prop}

\begin{preuve} 
V{\'e}rifions tout d'abord que $\uomm$ est bien d{\'e}fini.
Notons $\uet=\classet{\uk}\in\et{\End(V)}$.
On a $\Neoom(\uet)<\pinfty$.
Notons $\uett :\ \classet{\vk} \mapsto \classet{\uk(\vk)}$ 
 l'endomorphisme de $\Vet$ associ{\'e} {\`a} $\uet$, qui est
 clairement bien d{\'e}fini.

On note $\Neoomm$ la \pseudonorme{} sur $\End(\Vet)$
associ{\'e}e {\`a} la \pseudonorme{} $\normoom$ sur $\Vet$.
Montrons qu'on a
$\Neoomm(\uett) =  \Neoom(\uet)$.
En effet, pour tout $\vet$ de $\Vet$,
\begin{equation}\label{eq- maj norme de psi(u)}
\normoom\left( {\uett(\vet)} \right) \leq \Neoom(\uet)\normoom(\vet),
\end{equation}
donc
$\Neoomm(\uett) \leq \Neoom(\uet)$.
Soit $\vk\in V$ tel que
$\normo(\vk)=1$ et  $\normo(\vk)\geq \frac{1}{2}
\Neo(\uk)\normo(\vk)$. Si $\vet=\classet{\vk}\in\Vet$ on a
$\normoom (\uett(\vet) )  \geq \Neoom(\uet)\normoom(\vet)$,
 et donc $\Neoomm(\uett) \geq \Neoom(\uet)$.

Par cons{\'e}quent $\uett$ conserve $\Vetp$ et
$\Veto$. Donc $\uett$ passe bien au quotient en un endomorphisme
de $\Vom=\Vetp/\Veto$, qui est clairement $\uomm$.
De plus, si $\Neoom(\uet)=0$, alors $\uomm=0$ dans $\End(\Vom)$ par
(\ref{eq- maj norme de psi(u)}).  Donc $\uomm$ ne d{\'e}pend que de
$\classom{\uk}$ dans $\omq{\End(V)}$.
\end{preuve}

\begin{coro}
\label{coro- isom SLn(KKom)}
\begin{enumerate}

\item L'application 
\[
\psi :\begin{array}{ccl}
\omEndV & \fleche & \End(\Vom)\\
  \uom &\mapsto & \uomm
\end{array}
\]
est un isomorphisme de $\KKom$-alg{\`e}bres norm{\'e}es.

\item Elle induit un isomorphisme de groupes 
$ \psi:\ \omq{\SLn(\KK)}\fleche \SLn(\KKom)$.
\qed
\end{enumerate}
\end{coro}

\begin{rema}
1) Soit $\Gb$ un sous-groupe alg{\'e}brique de $\SLn$ d{\'e}fini sur
 $\KK$ et $G=\Gb(\KK)$. Comme $\KK\subset\KKom$ (via les suites
constantes), on peut consid{\'e}rer le groupe $\Gb(\KKom)$, et on a alors  
 $\psi(\omq{G})\subset\Gb(\KKom)$.
2) Soit $\Ga$ un groupe de type fini, et $S$ une partie g{\'e}n{\'e}ratrice de
$\Ga$. Soit $\rok :\Ga \fleche \Gb(\KK)$ une suite de repr{\'e}sentations
telles que  $N(\rok(s))^\lslak$ est born{\'e} pour tout  $s\in S$.  
Alors la suite $\rok$ induit une repr{\'e}sentation $\room :\Ga \fleche
\Gom$ d{\'e}finie par $\room(\ga)=\classom{\rok(\ga)}$ pour tout $\ga$ de $\Ga$.
Par la remarque pr{\'e}c{\'e}dente $\psi\circ\room$ est alors une
repr{\'e}sentation de $\Ga$ dans $\Gb(\KKom)\subset \SLn(\KKom)$.  
\end{rema}

\subsection{C{\^o}ne asymptotique de $\Es$ 
et immeuble de $\SLn(\KKom)$.}
\label{ss- conas et imm de BT}

Pour $\FF=\KK$ ou $\KKom$, on rappelle qu'on note 
$\Es(\FF)$ l'espace associ{\'e} {\`a} $\SLn(\FF)$, qu'on identifie
{\`a} $\BNl(\FF^n)$
(voir section \ref{ss- normes}). 
On notera aussi $\Es=\Es(\KK)$.
On note $(\Esom,\oom,\dom)$ le c{\^o}ne asymptotique de $(\Es, \normo,
\dk = \lslak d)$.
\begin{prop} 
On peut d{\'e}finir une action naturelle  de $\omq{\SLn(\KK)}$  
sur l'immeuble $\Esom$ (par
au\-to\-mor\-phis\-mes) par
\[(\classom{\gk},\classom{\normk}) 
\mapsto\classom{\gk\cdot\normk}\;.\]
\end{prop}

\begin{preuve}
Soit $(\gk)$ une suite dans $\SLn(\KK)$.
Par la formule (\ref{eq- d neuo et N(g)}) de la section \ref{ss-
  normes} on a que $\Neo(\gk)^\lslakh$ \omborne{} {\'e}quivaut {\`a}
$\lslak d(\normo , \gk \normo)$ \omborne. 
Alors $\gpom:\classom{\normk}\mapsto \classom{\gk\cdot\normk}$ 
d{\'e}finit bien  un automorphisme
de $\Esom$ (voir section \ref{ss- prel conas}). 

Il reste {\`a} voir que  si $\limom \Neo(\gk-\id)^\lslakh=0$ alors 
$\gpom=\id_{\Esom}$.
Soit  alors $\normom=\classom{\normk}$ dans $\Esom$. 
D'apr{\`e}s le lemme \ref{lemm- d neu, g.neu et N(g-I)}, on a
$$\lslak\dinf(\normk,\gk\normk)\leq \log\left(1+ e^{{2\over
    \lak}\dinf(\normo, \normk)} {\scriptstyle \max
  \left\{\Neo(\gk-\id)^\lslak, \Neo(\gk^{-1}-\id)^\lslak \right\}
}\right)$$
Or $\lslak \dinf(\normo, \normk)$ est
\omborne{} et, comme $\limom\Neo(\gk-\id)^\lslak = 0$, on a {\'e}galement que
$\limom\Neo(\gk^{-1}-\id)^\lslak = 0$.
Donc $\lslak
\dinf(\normk, \gk\normk)\tendom 0$, et $\gpom\normom =\normom$.
\end{preuve}

On note $\psi: \omq{\SLn(\KK)} \fleche \SLn(\KKom)$ 
l'isomorphisme 
du corollaire \ref{coro- isom SLn(KKom)}.

\begin{theo} \label{theo- modele alg Esom}
L'application 
$$\phi :\begin{array}{ccl}
\Esom(\KK) & \longrightarrow & \EsKom\\
\classom{\normk} &\mapsto &\limom\normk^\lslakh
\end{array}$$
d{\'e}finit un isomorphisme d'immeubles affines, $\psi$-{\'e}quivariant.
\end{theo}

\begin{preuve} 
Montrons tout d'abord que $\phi$ est bien d{\'e}finie.
Pour toute suite $\normet=\classet{\normk}$ dans $\Eset$
on peut d{\'e}finir $\normett:\Vet\mapsto[0,\pinfty]$ par
$\normett(\classet{\uk})=\limom \normk(\uk)^\lslakh$.
On voit facilement que $\normett$ est une \pseudonorme{} ultram{\'e}trique
sur $\Vet$.

Soient $\normet=\classet{\normk}$ et $\normpet=\classet{\normpk}$ dans
$\Eset$.
On consid{\`e}re la \pseudodistance{} 
$\dinfom$ sur $\Eset$ d{\'e}finie par
$\dinfom(\normet,\normpet)=\limom \lslak \dinf(\normk,\normpk)$. 
Elle est {\'e}quivalente {\`a} $\dom$ par la  formule (\ref{eq- d et dinf}).
On suppose que $\dinfom(\normet,\normpet)<\pinfty$.
Soit $\vet=\classet{\vk}\in\Vet$. Pour tout $k$, on a dans $\RR$ 
$$0\leq e^{-\dinf(\normk,\normpk)} \normpk(\vk) \leq \normk(\vk)
\leq e^{\dinf(\normk,\normpk)} \normpk(\vk)$$
Par cons{\'e}quent, on a dans $\RRet$ 
$$0\leq \classet{e^{-\dinf(\normk,\normpk)}}\ \classet{ \normpk(\vk)} \leq
\classet{\normk(\vk)} \leq \classet{e^{\dinf(\normk,\normpk)}}\ \classet{
\normpk(\vk)}$$
et en prenant la \pseudova{} $\abslaet{\ }$ de chaque
terme, on obtient dans $[0,\pinfty]$
\begin{equation}\label{eq- normes equiv}
 0\leq e^{-\dinfom(\normet,\normpet)} \normpett(\vet) \leq \normett(\vet)
\leq e^{\dinfom(\normet,\normpet)} \normpett(\vet)
\; . 
\end{equation}

Supposons maintenant que $\dinfom(\normet,\classet{\normo})<\pinfty$.
En appliquant (\ref{eq- normes equiv}) {\`a} $\normpet=\normoet$,
on voit que $\normett(v)<\pinfty$ 
si et seulement si $\normoett(v)<\pinfty$ 
et  que $\normett(v)=0$  si et seulement si
$\normoett(v)=0$. 
Donc $\normett$ passe bien au quotient en une norme
ultram{\'e}trique  sur $\Vom$, qu'on notera {\'e}galement $\normomm$.
La formule (\ref{eq- normes equiv}) implique {\'e}galement que si
$\dinfom(\normet,\normpet)=0$, alors $\normomm=\normpomm$.
 Par cons{\'e}quent $\normomm$ ne d{\'e}pend
que du point $\normom= \classom{\normk}$ de $\Esom$ d{\'e}fini par 
$\normet$. 

On a donc v{\'e}rifi{\'e} que l'application
$\classom{\normk}\mapsto\limom\normk^\lslakh$ est
 bien d{\'e}finie de $\Esom$ dans l'espace $\Normes(\Vom)$ des normes
ultram{\'e}triques sur $\Vom$.
Il est clair que $\phi$ est $\psi$-{\'e}quivariante.
Comme 
$\omq{\SLn(\KK)}$ agit transitivement sur les appartements marqu{\'e}s de
$\Esom$, 
et $\SLn(\KKom)$ agit transitivement sur les
appartements marqu{\'e}s de $\EsKom$,
il suffit maintenant de montrer que, si $\omfeo:\Aa\fleche\Esom$ est l'appartement
marqu{\'e} canonique $\ulim{f_\eo}$,
on a $\phi\circ\omfeo = \feoom$.

Soit $\af=(\afln)\in\Aa$. 
Notons $\normk=f_\eo(\lak\af)$.
Soit $\aom=(\alnom)$ dans $\KKom^n$ avec 
 $\aiom=\classKom{\aik}$.
Dans le cas o{\`u} $\KK=\RR$ ou $\CC$,  on a 
\begin{eqnarray*} 
\phi(\omfeo(\af))(\aom) 
&=& \absom{\classRom{ \normk(\alnk)}}
= \absom{\classRom{ \sqrt{\psum_{i=1}^n e^{\lak\afi}\abs{\aik}^2}\;}} 
\\
&=& \absom{\sqrt{\psum_{i=1}^n \classom{e^{\lak\afi}}
\ \classom{\abs{\aik}}^2}\;}
 =\max\limits_\ilan e^{\afi}\absom{\aiom}
=\feoom(\alpha)(\aom)
\end{eqnarray*} 
d'apr{\`e}s (\ref{eq- supp}).
Le cas o{\`u} $\KK$ est ultram{\'e}trique est clair.
\end{preuve}

\section{Vecteur de translation}
\label{s- VdT}

Dans cette section, on suppose que $\Es$ est un espace sym{\'e}trique
ou un immeuble affine complet (voir   section \ref{ss- prel ES IMM},
dont on reprend les hypoth{\`e}ses et notations).
On introduit le {\em  vecteur de translation} $v(g)$ 
d'un automorphisme 
$g$ de $\Es$ de mani{\`e}re g{\'e}o\-m{\'e}\-tri\-que, comme un raffinement 
de la  longueur de translation $\ell(g)$.
On {\'e}tablit ensuite des propri{\'e}t{\'e}s de continuit{\'e} de $v$, 
la continuit{\'e}  asymptotique (Proposition \ref{prop-
  passage au cone}).

\subsection{Vecteur de translation d'une isom{\'e}trie}

Soit $g\in\Aut(\Es)$. 
Si $\Min(g)\neq\emptyset$, ce qui
est toujours le cas si $\Es$ est un immeuble 
(th{\'e}or{\`e}me \ref{theo- immob}),
le type $\Cd_g(x)=\Cd(x,gx)$ du segment joignant $x\in\Min(g)$ {\`a} $gx$
est constant (car deux g{\'e}od{\'e}siques translat{\'e}es par $g$ sont parall{\`e}les
donc dans un m{\^e}me plat maximal).
Le {\em vecteur de translation} de $g$ est alors par d{\'e}finition 
$v(g)=\Cd(x,gx)$ pour un (tout) point $x\in\Min(g)$. 
On peut le d{\'e}finir plus g{\'e}n{\'e}ralement pour tout $g$ gr{\^a}ce au
r{\'e}sultat suivant.

\begin{prop} 
\label{prop- def vdt}
Soit $\Es$  un espace sym{\'e}trique et $g\in\Aut(\Es)$.  
L'adh{\'e}rence dans $\Cb$ de $\{\Cd(x, gx),\ x\in \Es\}$ contient un
unique segment de longueur minimale, qu'on note $v(g)$, et qu'on
appelle {\em vecteur de translation} de $g$.  
\end{prop}

\begin{proof}
On a  $\inf_{x\in\Es}\norme{\Cd(x,gx)}=\ell(g)$.
L'existence est claire, montrons l'unicit{\'e}.
Soit $g=heu$ la d{\'e}composition de Jordan de $g$ et $\eps>0$. Soit
$\eta$ donn{\'e} par le lemme \ref{lemm- d_g et distance a F_g}.
Si $x\in\Es$ est tel que $d_g(x) \leq \ell(g) + \eta$
et si $y$ est
la projection  de $x$ sur $\Min(h)$, alors  la
proposition \ref{prop- Cd 1-Lip} entra{\^\i}ne que
$\norme{\Cd_g(x)-\Cd_g(y)} \leq 2d(x,y)\leq 2\eps$ et
$\norme{\Cd_g(y)-\Cd_h(y)} \leq d(gy, hy) \leq \eps$.
On en d{\'e}duit que
$\norme{\Cd_g(x)-\Cd_h(y)} \leq 3\eps$,
ce qui conclut car $\Cd_h(y)$ ne d{\'e}pend pas de $y\in\Min(h)$.
\qed\end{proof}

\begin{remas*}
  \begin{enumerate}
  \item 
  La notion de vecteur de translation est plus fine que celle de
longueur de translation (car $\ell(g)=\norme{v(g)}$). 

\item  $v$ est invariant par conjugaison, et pour tout $g$, on a $v(g^k)=kv(g)$ pour tout $k\in\NN$.
\item Dans le cas des espaces sym{\'e}triques, 
on a $v(g)=v(h)$ o{\`u} $h$ est la partie hyperbolique dans la
  d{\'e}composition de Jordan de $g$ (d'apr{\`e}s la preuve de la proposition
\ref{prop- def vdt}).
Cet invariant co{\"\i}ncide donc avec la {\em projection de Jordan}
  de $g$ (i.e. la projection de $h$ sur $\Cb$, voir par exemple 
\cite[1.1 et 2.4]{Benoist}).
\end{enumerate}
\end{remas*}

\begin{exem*}
Dans le cas o{\`u} $\Es$ est l'espace sym{\'e}trique associ{\'e} {\`a} $G=\SLn(\RR)$, on peut
choisir \(\Cb=\{
(\alpha_1,\ldots,\alpha_n) \in \RR^n\sum_{i=1}^n \alpha_i = 0,\ \alpha_1 \geq \ldots \geq
\alpha_n\}\) et pour $g\in G$
de valeurs propres $\la_1(g)$, \ldots, $\la_n(g)$ on a alors
\[v(g)=(\log\abs{\la_1(g)},\ldots,\log\abs{\la_n(g)})\]
\end{exem*}

\paragraph{Propri{\'e}t{\'e}s fonctorielles.}
Si $\Es$ est euclidien, alors  on a $\Cb=\Aa=\Es$, et 
tout $g\in\Aut(\Es)$ est une translation de vecteur $v(g)$.

Supposons que $\Es=\Es_1\times \Es_2$ est le produit de deux espaces
sym{\'e}triques ou immeubles affines $\Es_i$. 
Soit $G_i=\Aut(\Es_i)$.
Soient $\Aa_1$, $\Aa_2$ les plats maximaux
de $\Es_1$ et $\Es_2$ tels que $\Aa=\Aa_1\oplus \Aa_2$,   
et  $\Cb_1$ et $\Cb_2$ les chambres de Weyl ferm{\'e}es de $\Aa_1$ et
$\Aa_2$ telles que $\Cb=\Cb_1+\Cb_2$. 
Soit $v_i:G_i\fleche\Cb_i$ la fonction vecteur de translation de
$\Es_i$.
 Alors   tout automorphisme $g$ de $\Es$ est de la
forme $(g_1,g_2)$ o{\`u} $g_i\in G_i$, et, d'apr{\`e}s la section \ref{ss- type - produit, sous-espace}, on a 
\begin{equation}
\label{eq- vdt du produit}
v(g) =v_1 (g_1)+v_2 (g_2).
\end{equation}

Soit $\Es'$  un sous-espace sym{\'e}trique ou un sous-immeuble de $\Es$.
Soit $G'=\Aut(\Es')$.
Soit  $\Cb'$ une chambre de Weyl
ferm{\'e}e de $\Es'$.
Soit $v_{\Es'}: G' \fleche\Cb'$ la fonction
vecteur de translation de $\Es'$.
Si $g\in\Aut(\Es)$ pr{\'e}serve $\Es'$ et $g'=g_{|\Es'}$ est un
automorphisme de $\Es'$, alors d'apr{\`e}s la section \ref{ss- type -
produit, sous-espace} on a
\begin{equation}
\label{eq- vdt d'un sous-esp}
v(g) =\type(v_{\Es'}(g'))
\end{equation}

\begin{prop} 
\label{prop v continue}
La fonction $v:G\fleche\Cb$ est continue.  
\end{prop}

\begin{remas*}
  \begin{enumerate}
\item 
Ceci entra{\^\i}ne que, lorsque $\Es$ n'a pas de facteur euclidien, la
longueur de translation est continue sur $\Isom(\Es)_0$ .  Cette
propri{\'e}t{\'e} n'est pas vraie pour tous les espaces m{\'e}triques
$\CAT(0)$. Par exemple, dans le plan euclidien, les translations non
nulles sont limites de rotations.

\item Pour $G=\SLn(\RR)$, il s'agit de voir que la liste des modules
  des valeurs propres, en ordre d{\'e}croissant, d{\'e}pend
  contin{\^u}ment de la matrice, ce qui est bien connu.
  \end{enumerate}
\end{remas*}

Nous allons donner maintenant une d{\'e}monstration purement
g{\'e}o\-m{\'e}\-tri\-que de la continuit{\'e} de $v$, 
reposant sur la dynamique des isom{\'e}tries sur $\bordinf\Es$.

\begin{proof} %
Soit $\gk\tend  g$ dans $G$. 
Montrons que quitte {\`a} extraire $v(\gk)\tend v(g)$.

Si $\ell(g)=0$, alors, par semi-continuit{\'e} sup{\'e}rieure de $\ell$, on a
$\ell(\gk)\tend 0$ et donc $v(\gk)\fleche 0$, ce qui conclut.

Dans le cas o{\`u}  $d(\Min (\gk),\xo)$ reste born{\'e}, 
ce qui est toujours le cas quand $\Es$ est un
immeuble affine  (th{\'e}or{\`e}me \ref{theo- immob}), 
cela d{\'e}coule ais{\'e}ment du lemme \ref{lemm- Min(gk) reste a distance
  bornee} et de la proposition \ref{prop- type_gk cv vers type_g}.

On peut donc supposer que $\Es$ est un espace sym{\'e}trique et que
$\ell(g)>0$. Cela d{\'e}coule alors du r{\'e}sultat plus fort suivant.

\begin{lemm}
\label{lemm- continuite axes faible}
Soit $\Es$ un espace sym{\'e}trique et $\gk\tend g$ dans $\Aut(\Es)$ avec $\ell(g)>0$.
On note $\xipp$ et $\xim$ les points fixes attractif et r{\'e}pulsif de
$g$ dans $\bordinf\Es$ 
(cf. remarque \ref{rema- pf attractif ES et IMM}) et $F=\Fximp$.
Alors pour $k$ suffisamment grand il existe $\xippk$ et $\ximk$
oppos{\'e}s dans $\bordinf\Es$, de m{\^e}me types que $\xipp$ et $\xim$, tels
que

\begin{enumerate}
\item 
\label{item- xippk tend xipp et de meme type}
$\xippk\tend\xipp$ et $\ximk\tend \xim$.

\item 
\label{item- Fk tend Fg} 
En notant $\Fk=F_{\ximk\xippk}$,
  on a $\Fk \tend F$ pour la topologie de Hausdorff point{\'e}e.

\item
\label{item- decomposition de gk}
$\gk$ fixe $\xippk$ et $\ximk$, et, si $\gk=(\gpk, \tk)$  sur $\Fk=\Ck\times
\RR$, on a $\ell(\gpk)\tend 0$ et $\tk\tend \ell(g)$.

\item 
\label{item- vgk tend vg}
$\vgk\tend v(g)$.

\item 
\label{item- xippgk et xippk}
En notant
  $\xippgk$ le point fixe attractif de
$\gk$ dans $\bordinf\Es$, on a
 $\angleTits(\xippk,\xippgk)\tend 0$, $\xippgk\in\bordinf\Fk$ et
$\xippk$ est l'unique point de
de m{\^e}me type que $\xipp$ dans l'adh{\'e}rence de la facette ouverte de
  $\bordinf\Es$ contenant $\xippgk$.

\end{enumerate}
En particulier, si $v(g)$ est r{\'e}gulier alors 
$\Min(\gk) \tend\Min(g)$.
\end{lemm}
\begin{proof}[Lemme \ref{lemm- continuite axes faible}]
Soit $p\in F$ et $D>0$.  
Il existe   une boule topologique  ferm{\'e}e $B$ telle que
$\xipp\in\interieur{B}$ et $B\subset\Ombm pD$
  (voir section \ref{ss- prel ES - Gxi variete}).
D'apr{\`e}s la proposition \ref{prop- dynamique}, il existe $N\in\NN$ tel
que $g^N(\Ombm p D )\subset\interieur{B}$.

Notons $\gkN=g_k^N$ et $\gN=g^N$.  On a $\gkN\tend\gN$ et  
$\ell(\gN)=N\ell(g)>0$, $\FgN=\Fg=F$ et $\xipp$ et $\xim$ sont 
les points fixes attractif et r{\'e}pulsif de $\gN$. 
On a $\gN(B)\subset\interieur{B}$.
Comme l'action de $G=\Aut(\Es)$ sur $G\xipp$ est continue,
il existe $K\in\NN$ tel que pour $k\geq K$, on a $\gkN(B)\subset
B$. Comme $\gkN$ est continue sur $B$, elle admet alors un point fixe
$\xippk$ dans $B$ (donc de m{\^e}me type que $\xipp$) par le th{\'e}or{\`e}me de Brouwer.
Quitte {\`a} extraire $\xippk \tend z$ dans $B$, fix{\'e} par $\gN$.  
Comme $\xipp$ est le seul point fixe de $\gN$ dans $B$
(proposition \ref{prop- dynamique}), on a $z=\xipp$.
De m{\^e}me, il existe $\ximk\in\bordinf\Es$  fix{\'e} par $\gkN$, de m{\^e}me type
que $\xim$, tel que $\ximk\tend \xim$. 
La proposition \ref{prop- points visuellement presque   opposes}
entra{\^\i}ne alors que $\Fk\tend F$.

On note $(\gp,t)$ la d{\'e}composition de $g$ sur $F=C\times\RR$. Alors
$t=\ell(g)$ et $\ell(\gp)=0$  (proposition \ref{prop- prel ES - l(g)=l(h)}). 
La d{\'e}composition de $\gN$ sur $F=C\times\RR$ est alors 
$(\gNp,\tN)$ avec  $\gNp=(\gp)^N$ et $\tN=Nt$. 
En notant $(\gkNp,\tkN)$ la d{\'e}composition de $\gkN$ sur
$\Fk=\Ck\times\RR$, on a alors que
$\gkNp\tend\gNp$ et $\tkN\tend \tN$.
Donc $\ell(\gNp)=N\ell(\gp)=0$  
et $\limsup\ell(\gkNp)\leq\ell(\gNp)$, donc 
$\ell(\gkNp)\tend 0$.

Nous pouvons maintenant montrer le point \ref{item- vgk tend vg}.
Soit $\Cb_k$ une chambre de Weyl ferm{\'e}e de $\Fk$ et $\Cb_k'$ la chambre
de Weyl ferm{\'e}e de $\Ck$ telle que $\Cb_k=\Cb_k'\times \RR$.
Notons $v_{\Fk}:\Aut(\Fk)\fleche \Cb_k$ et $v_{\Ck}:\Aut(\Ck)\fleche \Cb_k'$
les fonctions vecteur de translation.
D'apr{\`e}s (\ref{eq- vdt du produit}) et (\ref{eq- vdt d'un sous-esp}),
on a $v_{\Fk}(\gkN)=v_{\Ck}(\gkNp)+\tkN$ et $v(\gkN)=\type(v_{\Fk}(\gkN))$. 
Comme $\type(\tkN) \tend v(\gN)$ et $v_{\Ck}(\gkNp) \tend 0$,
on en d{\'e}duit que $v(\gkN) \tend v(\gN)$, car $\type$ est continue.
Donc $\vdt(\gk)=\frac{1}{N}\vdt(\gkN)\tend \frac{1}{N}\vdt(\gN)=\vdt(g)$.
Pour $k$
suffisamment grand, le point fixe attractif $\xippgk$ de $\gk$ existe car
$\ell(\gk)>0$ (par ce qui pr{\'e}c{\`e}de) (cf. remarque \ref{rema- pf attractif ES et IMM}), et il co{\"\i}ncide avec celui de $\gkN=\gk^N$.
  Les deux premi{\`e}res assertions du point \ref{item-
  xippgk et xippk} d{\'e}coule de la proposition \ref{prop- faisceau
  stable et point attractif} (appliqu{\'e}e {\`a} $\gkN$), car
  $\ell(\gkNp)\tend 0$. 
La troisi{\`e}me
assertion en d{\'e}coule par les propri{\'e}t{\'e}s des immeubles sph{\'e}riques. 
En effet, notons $\untype$ le type de $\xipp$. 
Le type de $\xippgkN$ est $\untype_k=\bord\vdt(\gk)$,
qui tend vers $\bord\vdt(g)=\untype$ dans $\Cb$.
 Donc pour $k$ assez grand
$\untype$ est dans l'adh{\'e}rence de la facette ouverte de $\Cb$
contenant $\untype_k$.
Il existe donc  un unique point $\afk$ de
de type $\untype$ dans l'adh{\'e}rence de la facette ouverte de
  $\bordinf\Es$ contenant $\xippgk$.
Alors
$\angleTits(\xippgk,\afk)=\angle(\untype_k,\untype)\tend 0$, 
donc $\angleTits(\afk,\xippk)\tend 0$, et $\angleTits(\afk,\xippk)$
est dans l'ensemble fini $D(\untype,\untype)$, donc $\afk=\xippk$ pour
$k$ assez grand.

Il ne reste qu'{\`a} montrer le \ref{item- decomposition de gk}.
Comme $\gk$ fixe l'adh{\'e}rence de la facette contenant $\xippgk$, on
obtient -enfin- que $\gk$ fixe $\xippk$. De m{\^e}me, on a que $\gk$ fixe
$\ximk$. 
On a alors $\gkNp=(\gpk)^N$ et $\tkN=N\tk$, ce qui conclut.
\qed\end{proof} %
\end{proof} %

\subsection{Continuit{\'e} asymptotique du vecteur de translation}
\label{ss- vdT passage au cone}
Soit $\om$ un ultrafiltre sur $\NN$ plus fin que le filtre de Fr{\'e}chet.
Soit $(\lak)_{k\in\NN}$ une suite de r{\'e}els positifs tendant vers
$\pinfty$, et $(\ok)_k$ une suite de points de $\Es$. Soit $(\Esom,\oom,\dom)$
le c{\^o}ne asymptotique de $(\Es,\ok,\lslak d)_k$ suivant l'ultrafiltre
$\om$ (voir section \ref{ss- prel conas}).

\begin{prop} \label{prop- passage au cone} 
Soit $(\gk)$ une suite dans $G$ telle que $\limom\lslak
d(\ok, \gk \ok)<\pinfty$.
Soit $\gom=[\gk]$ l'isom{\'e}trie 
de $\Esom$ associ{\'e}e.
Alors  $\lslak\vgk\tend_\om\vgom$ dans $\Cb$.
\end{prop}

\begin{preuve} 
Commen\c cons par le cas o{\`u} $\lslak d(\Min (\gk),\xo)$ reste born{\'e}, 
ce qui est toujours le cas quand $\Es$ est un
immeuble affine complet
d'apr{\`e}s le th{\'e}or{\`e}me \ref{theo- immob},
qui se traite simplement 
de mani{\`e}re analogue {\`a} la preuve de la continuit{\'e} de $v$.
Alors $\limom\lslak d(\Min(\gk),\ok)<\pinfty$.
Soit $\xk$ dans $\Min(\gk)$ tel que $\limom \lslak d(\xk,\ok)
<\pinfty$.
Soit $\xom=[\xk]$. Le lemme \ref{lemm- Min(gk) reste a
distance ombornee} entra{\^\i}ne que $\xom\in\Min(\gom)$. 
Par cons{\'e}quent $\Cd_\gom(\xom)=v(\gom)$. 
D'apr{\`e}s la proposition \ref{prop- type asympt continu}, on a
$\Cd_\gom(\xom)=\limom \lslak\Cd_{\gk}(\xk)$, 
donc $\vgom=\limom \lslak \vgk$.

Supposons maintenant que $\Es$ est un espace sym{\'e}trique. 
Comme $\Esom$ est un immeuble affine complet
(th{\'e}or{\`e}me \ref{theo- cone as = immeuble}), on sait que 
$\gom$ fixe un point ou translate une g{\'e}od{\'e}sique 
(th{\'e}or{\`e}me \ref{theo- immob}).
On voit facilement que $\limom \lslak\lgk \leq \lgom$.  Par cons{\'e}quent,
si $\gom$ fixe un point, on a $\limom \lslak\lgk=\lgom=0$, donc, comme
$\norme{\vgk}=\lgk$ pour tout $k$, on a $\limom \lslak\vgk=0$, ce
qui conclut.
On peut donc supposer dor{\'e}navant que $\gom$ translate non trivialement
une g{\'e}o\-d{\'e}\-si\-que $\rom$ de $\Esom$.
Quitte {\`a} conjuguer chaque $\gk$, on peut supposer pour simplifier 
que la suite  $\ok$ est constante {\'e}gale {\`a}
$\xo$. 
Soit $r$ une g{\'e}od{\'e}sique de $\Es$ telle que $r(0)=\xo$, de m{\^e}me type
que $\rom$. Alors il existe une suite $(\hk)$ dans $G$ telle que
$\limom\lslak d(\xo, \hk\xo)<\pinfty$ et $h_\om=[\hk]$ envoie $[r]:\;
t\mapsto [r(\lak t)]$ sur $\rom$ (proposition \ref{prop- Gom geod type
  constant}). Quitte {\`a} conjuguer chaque $\gk$ par $\hk$, on peut donc
supposer  que $\rom=[r]$.
Notons $\xipp=r(\pinfty)$, $\xim=r(\minfty)$.
\begin{lemm}[Contraction des ombres] 
\label{lemm- ombres contractees}
Pour tous  $D,t,\eps$ dans $]0,\pinfty[$,  on a  
\[\gk(\Ombm{r(0)}{D})\subset\Ombm{r(t)}{\eps} \mbox{\ pour \omptt{} $k$}\; .\]
\end{lemm}

\begin{preuve}
On raisonne par l'absurde. Sinon, il existe $D,t,\eps >0$ et 
une suite $\zk$ dans $\bordinf\Es$ avec $\zk\in \Ombm {r(0)}D$ et
$\gk \zk\notin \Ombm {r(t)}\eps$ \ompp. Soit $\sigk$ une g{\'e}od{\'e}sique de
$\xim$ {\`a} $\zk$ telle que $d(r(0),\sigk(0))\leq D$.  Alors les
g{\'e}od{\'e}siques $\rom$ et  $\sigom
=[\sigk]$ de $\Esom$ sont confondues sur $]{\minfty},0]$, car elles sont
asymptotes en $\minfty$ et $\dom(\rom(0),\sigom(0))=\limom \lslak
d(r(0),\sigk(0))=0$.
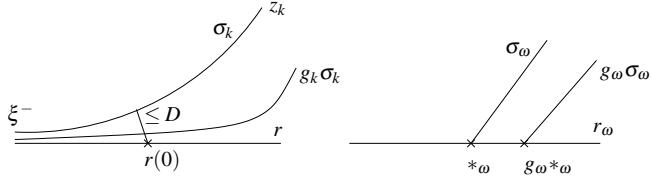
\begin{figure}[h]
\centering
\input{lemme_contrac_ombres.pstex_t}
\end{figure}

On a donc $\angle_{\oom}(\rom(\pinfty),\ \gom\sigom(\pinfty))=0$.
Donc, comme le passage au c{\^o}ne asymptotique augmente les angles
(proposition \ref{prop- passage des angles au conas}),
 on a $ \angle_{r(t)}(\xipp,\gk \zk) \tendom 0$
(car $[r(t)]=\oom$).
Par cons{\'e}quent $ \angle_{r(t)}(\xim,\gk \zk) \tendom \pi$, et donc
  $d(r(t),F_{{\xim,\gk \zk}})\tendom 0$ 
par la proposition \ref{prop- points visuellement presque opposes}.
Or par hypoth{\`e}se  $d(r(t),F_{{\xim,\gk \zk}}) \geq \eps$ pour
\omptt{} $k$, contradiction.\end{preuve}

\medskip

\noindent{\em Fin de la preuve de la proposition \ref{prop- passage
au cone}.}
Soit $D>0$. 
 L'ombre $\Ombm {r(0)} D$ est un voisinage de $\xipp$ dans son orbite
$G\xipp$. Comme $G\xipp$ est une vari{\'e}t{\'e}, elle contient donc une boule
topologique euclidienne $B$ contenant $\xipp$ dans son int{\'e}rieur. Comme
les ombres $\Ombm {r(t)}\eps$, ${t,\eps>0}$, forment une base de
voisinages de $\xipp$ dans $G\xipp$ 
(proposition \ref{prop- ombres forment base de vois}),
il existe $t>0$ et $\eps>0$ tels que  ${\Ombm
{r(t)}\eps}\subset B$.
D'apr{\`e}s le lemme \ref{lemm- ombres contractees}, on a alors, pour
\omptt{} $k$, que $\gk(B)=B$ donc que $\gk$ admet un
point fixe $\xippk$ dans ${\Ombm {r(t)}\eps}$ par le th{\'e}or{\`e}me de Brouwer (car
$\gk$ est continue sur $B$).

Comme $\xippom=[\xippk]$ est un point de $\bordinf\Esom$ oppos{\'e} {\`a}
$[\xim]=\rom(\minfty)$ et fix{\'e} par $\gom$, qui translate
$\rom$, il est n{\'e}cessairement {\'e}gal {\`a}
$\rom(\pinfty)=[\xipp]$ (proposition \ref{prop- points fixes isom axiale}).
On peut donc supposer, quitte {\`a} conjuguer chaque $\gk$, que
$\xippk=\xipp$, c'est-{\`a}-dire que $\gk$ fixe $\xipp$ pour \omptt{} $k$.
On peut aussi supposer (en {\'e}changeant les r{\^o}les de $\xipp$ et $\xim$)
que $\gk$ fixe $\xim$ pour \omptt{} $k$.

Pour \omptt{} $k$, on a alors que $\gk$ pr{\'e}serve le faisceau
$\Fximp=F_r=C_r\times r$ et $\gk$ agit dessus comme $(\gpk,t_k)$
(voir section \ref{ss- CAT(0)}).
Soit $\pk=\gk\xo$ et $\qk$ sa projection  sur
$r$. Alors $\qk=t_k\xo$.
\begin{figure}[h]
\centering
\input{prop_v2cone2.pstex_t}
\end{figure}
Par 
la proposition \ref{prop- passage des angles au conas}, 
on a que $q_\om=[\qk]$ est  la projection  de $\gom \oom$ sur $\rom$.
Or $\gom \oom\in\rom$, donc on a $q_\om=\gom \oom$. 
Par cons{\'e}quent $\lslak d(\xo,\gpk\xo)\tendom 0$, et  donc $\limom
\lslak \ell(\gpk)=0$.
On en d{\'e}duit que $\limom \lslak\Cd(\xo,t_k\xo)=\Cd(\oom,\gom \oom)=\vgom$.

Soit $\Cb_F$ une chambre de Weyl ferm{\'e}e de $F=F_r$ et $\Cb_C$ la chambre de Weyl ferm{\'e}e de
$C=C_r$ telle que $\Cb_F=\Cb_C\times \RR$.
Notons $v_F:\Aut(F)\fleche \Cb_F$ et $v_F:\Aut(C)\fleche \Cb_C$
les fonctions vecteur de translation.
D'apr{\`e}s (\ref{eq- vdt du produit}) et (\ref{eq- vdt d'un sous-esp}), on a 
$v_F(\gk)=v_C(\gpk)+t_k$ et $\vgk=\type(v_F(\gk))$.
Or on vient de voir que $\limom \lslak v_C(\gpk)=0$ et que $\limom
\type(\lslak t_k)=\vgom$, donc, $\type$ {\'e}tant continue, on a $\lslak\vgk
\tendom \vgom$.
\end{preuve}

\begin{rema}
On a en fait d{\'e}montr{\'e} de plus les propri{\'e}t{\'e}s suivantes.
On suppose que $\ell(\gom)> 0$. Soit $\rom$ un axe de
 $\gom$ et $\ximom$ et $\xippom$ ses
extr{\'e}mit{\'e}s.
On note (pour $k$ suffisamment grand) $\xippgk$ et $\ximgk$ les points
fixes attractif et r{\'e}pulsif de $\gk$ dans $\bordinf\Es$ (cf. remarque
\ref{rema- pf attractif ES et IMM}), et $\xippk$ (resp. $\ximk$)
l'unique point  de type  $\vgom$ (resp.  $\vgom^\opp$) dans
l'adh{\'e}rence de la facette ouverte de $\bordinf\Es$ contenant $\xippgk$
(resp. $\ximgk$).  On a :
\begin{enumerate}
\item  $\xippom=[\xippgk]$ et $\ximom=[\ximgk]$.

\item 
Il existe
une suite \ombornee{} de g{\'e}od{\'e}siques $\rk$ de $\ximk$ {\`a} $\xippk$ telle
que  $\rom=\classom{\rk}$. 
En particulier $\gk\rk$ est parall{\`e}le {\`a} $\rk$ et $\lslak
d(\rk,\gk\rk)\tendom 0$.
\end{enumerate}
En particulier, 
si $\vgom$ est r{\'e}gulier alors 
$[\Min(\gk)]=\Min(\gom)$.
\end{rema}

\section{Compactification}
\label{s- compactification}

 Soit $\Ga$ un groupe infini de type fini, discret.
Soit $\KK$ un corps local (notons que $\KK$ est ou bien $\RR$ ou
$\CC$, ou bien ultram{\'e}trique).
On suppose que $G$ est un groupe r{\'e}ductif sur $\KK$ et que $\Es$ est
l'espace associ{\'e}, c'est {\`a} dire qu'on reprend les hypoth{\`e}ses et
notations  de la section \ref{ss- espace associe a un GR}.
On note $\xo$ un point-base de $\Es$. 

\subsection{Espaces de repr{\'e}sentations et quotient}
\label{ss- espace de rep et quotient}
On consid{\`e}re l'ensemble, not{\'e} $\R$ ou $\RGaG$ des repr{\'e}sentations de
$\Ga$ dans $G$, muni de la topologie de la convergence simple.
Si $S$ est une partie g{\'e}\-n{\'e}\-ra\-tri\-ce finie de $\Ga$,
l'espace $\R$ s'identifie {\`a} un sous-ensemble ferm{\'e}
de $G^S$ par l'application $\rho\mapsto(\rho(s))_{s\in S}$.
L'espace $\R$  est  m{\'e}trisable, {\`a} base d{\'e}nombrable, localement compact.
Le groupe  $G$ agit con\-ti\-n{\^u}\-ment sur $\R$ par
conjugaison.

L'espace topologique quotient usuel $\RsG$ n'{\'e}tant en g{\'e}n{\'e}ral pas
s{\'e}par{\'e}, 
on le remplace par son plus gros quotient s{\'e}par{\'e} $\Xsep=\RssG$, 
qui peut {\^e}tre d{\'e}crit, comme nous allons le voir, soit comme
le quotient de $\R$ par une relation d'{\'e}quivalence naturelle $\sim$,
soit plus explicitement comme une r{\'e}traction sur un sous-espace
naturel $\Rcr/G$ de $\RsG$ (la {\em semisimplification}).

\paragraph{Repr{\'e}sentations compl{\`e}tement r{\'e}ductibles.}
On dit que $\rho \in \R$ est {\em com\-pl{\`e}\-te\-ment r{\'e}\-duc\-ti\-ble}
(cr) si, pour tout $\alpha\in \bordinf \Es$ fix{\'e} par $\rho$, il existe
$\beta\in\bordinf \Es$, oppos{\'e} {\`a} $\alpha$, {\'e}\-gal\-e\-ment fix{\'e} par
$\rho$.
Cette notion, introduite par J.P.~Serre \cite{SerreCR} pour les
actions sur des immeubles sph{\'e}riques et les groupes alg{\'e}briques r{\'e}ductifs, est peu
restrictive.  
On renvoie {\`a} \cite{ParRepcr} pour
une {\'e}tude g{\'e}o\-m{\'e}\-trique et d'autres caract{\'e}risations naturelles.
Du point de vue alg{\'e}brique cela revient {\`a} dire que pour tout
sous-groupe parabolique $P$ de $G$ contenant $\rho(\Ga)$, il existe un
sous-groupe de Levi de $P$ contenant $\rho(\Ga)$ \cite[3.2.1]{SerreCR}.
Les repr{\'e}sentations Zariski-denses sont cr.
Dans le cas o{\`u} $G=\GLn\KK$, une repr{\'e}sentation est cr si et seulement
si l'action lin{\'e}aire sur $\KK^n$ associ{\'e}e est semi-simple.
Si la caract{\'e}ristique du corps $\KK$ est nulle, alors 
$\rho$ est cr si et seulement si  la composante neutre  de
l'adh{\'e}rence de Zariski de $\rho(\Ga)$ est un groupe r{\'e}ductif
\cite[Proposition 4.2]{SerreCR}.
Cela {\'e}quivaut par \cite{Richardson} {\`a} demander que l'orbite
$\Gb\cdot\rho$ soit Zariski-ferm{\'e}e dans la vari{\'e}t{\'e} alg{\'e}brique
$\R(\Ga,\Gb)$,
 ou encore par \cite{Bremigan} que l'orbite $G\cdot\rho$ soit ferm{\'e}e
 dans $\RGaG$.

\paragraph{Plus gros quotient s{\'e}par{\'e}.}
On note l'espace $\Rcr$ le sous-espace de $\R$
form{\'e} des re\-pr{\'e}\-sen\-ta\-tions cr, qui est stable sous $G$.  On
note $\RcrsG$ l'espace quotient, muni de la topologie quotient (dont
on rappelle qu'elle co{\"\i}ncide avec la topologie induite par l'inclusion
dans l'espace topologique quotient $\R/G$, voir la proposition 10 de
\cite[III, {\S} 2]{BouTG}). Cet espace est {\`a} base d{\'e}nombrable car $\R$
l'est.

Le r{\'e}sultat suivant, qui montre que $\RcrsG$ est le plus gros
quotient s{\'e}par{\'e} de $\R$ sous $G$, 
est bien connu pour $\KK=\CC$, et pour $\KK=\RR$,
il r{\'e}sulte de \cite{Richardson}, \cite{Luna} et \cite{RiSl}. Pour
$\KK$ de caract{\'e}ristique $0$ il r{\'e}sulte de \cite{Bremigan}. On en
trouvera une d{\'e}monstration directe, qui couvre aussi le cas de la
caract{\'e}ristique non nulle, dans \cite{ParRepcr}, auquel on renvoie
pour plus de d{\'e}tails.

\begin{theo}

\begin{enumerate}
\item L'espace $\RcrsG$ est s{\'e}par{\'e} et localement compact.

 \item  Toute orbite de $G$ dans $\R$ contient dans son
adh{\'e}rence une unique orbite cr $\ssO(\rho)$ ({\em
  semisimplification}).

\item La propri{\'e}t{\'e} $\ov{G\cdot\rho}\cap \ov{G\cdot\rho'}\neq \emptyset$
  d{\'e}finit une relation d'{\'e}quivalence $\rho\sim\rho'$ sur $\R$.
La projection $G$-invariante $\ssO
:\R \fleche \RcrsG$ est continue, et induit un ho\-m{\'e}o\-mor\-phisme $h$ de
l'espace quotient $\RssG:=\R/\sim$ sur $\RcrsG$, tel que
$\ssO=h\circ\psep$, o{\`u} $\psep:\R\fleche \RssG$ est la projection
canonique.
En particulier, $\RssG$ (resp. $\RcrsG$) est le plus gros quotient s{\'e}par{\'e} de
$\R$ sous $G$ (toute application continue $G$-invariante $f$ de
$\R$ vers un espace s{\'e}par{\'e} factorise {\`a} travers $\psep$).
\end{enumerate}
\end{theo}

On notera au passage que la topologie quotient issue de la
semisimplification $\ssO$ (utilis{\'e}e dans la litt{\'e}rature)
co{\"\i}ncide avec la topologie quotient usuelle de $\RcrsG$.
On identifiera dor{\'e}navant $\RssG$ et $\RcrsG$ (et les projections
$\psep$ et $\ssO$) via $h$. 
On notera cet espace $\XsepGaG$ ou $\Xsep$.
On notera $[\rho]$ la classe de
 $\rho\in\R$ pour $\sim$.
Remarquons que l'image dans $\Xsep$ d'un ferm{\'e} $G$-stable $F$ de $\R$
est un ferm{\'e} (car $F$ est stable par la semisimplification $\ssO$).
\paragraph{Lien avec le quotient alg{\'e}bro-g{\'e}om{\'e}trique.}
En ca\-rac\-t{\'e}\-ris\-ti\-que nulle,
il est {\'e}galement classique
de travailler avec le {\em quotient alg{\'e}bro-g{\'e}om{\'e}trique}
$\R\sslash{\Gb}$, qui est d{\'e}fini comme l'image canonique de $\R$
dans   le quotient alg{\'e}bro-g{\'e}om{\'e}trique
$\R(\Ga,\Gb)\sslash{\Gb}$ de la vari{\'e}t{\'e} alg{\'e}brique affine
$\R(\Ga,\Gb)$ sous l'action de $\Gb$. 

L'espace $\Xsep=\RssG$ consid{\'e}r{\'e} dans cet article est un peu ``plus
gros'' que $\R\sslash{\Gb}$, dans
le sens suivant.
On peut identifier $\R(\Ga,\Gb)\sslash{\Gb}$ {\`a} l'espace
des $\Gb$-orbites Zariski-ferm{\'e}es dans $\R(\Ga,\Gb)$, et $\Xsep$ {\`a} l'espace
des $G$-orbites ferm{\'e}es dans $\R$. 
La projection canonique $t: \R \fleche \R\sslash{\Gb}$ 
factorise en l'application naturelle 
$P: G\cdot\rho \mapsto \Gb\cdot\rho$  
de $\Xsep$ dans $\R\sslash{\Gb}$.
  Cette application (continue et surjective) est ferm{\'e}e et {\`a} fibres
  finies (voir \cite{Luna}, \cite{RiSl} pour $\KK=\RR$ et
  plus g{\'e}n{\'e}ralement \cite[5.17]{Bremigan} en caract{\'e}ristique nulle).

  \paragraph{Action de $\Out(\Ga)$}
Le groupe $\Aut (\Ga)$ des automorphismes de $\Ga$ agit {\`a} droite, par
pr{\'e}\-com\-po\-si\-tion  sur l'espace
$\R$.
Cette action est continue et commute avec l'action de $G$,
 donc induit une action 
de $\Aut(\Ga)$ sur $\Xsep$.
Le sous-groupe $\Int(\Ga)$ des automorphismes int{\'e}rieurs de $\Ga$ agit
trivialement sur $\Xsep$.
On obtient donc une action (par hom{\'e}omorphismes) du groupe $\Out(\Ga)
=\Aut(\Ga) /\Int(\Ga)$
sur $\Xsep$.

\bigskip

On va maintenant construire une compactification de $\Xsep$,
reposant sur la notion de vecteur de translation d{\'e}velopp{\'e}e en section
\ref{s- VdT}. 

\subsection{Spectre marqu{\'e} des vecteurs de translation}
\label{ss- spectre vdt}
On fixe un plat maximal $\Aa$ de $\Es$, de
groupe de Weyl vectoriel $\Wvect$, et une chambre de Weyl ferm{\'e}e $\Cb$ de $\Aa$.
Le groupe $G$ agit sur $\Es$ par automorphismes, 
donc le vecteur de translation $v : G \fleche \Cb$ est bien d{\'e}fini
(voir section \ref{s- VdT}).

\begin{rema*}[Remplacement par les longueurs]
 Dans tout ce qui suit, on peut remplacer 
$v:G\fleche\Cb$ par la  longueur de
translation $\ell=\norme{v}:G\fleche \RRp$ 
dans $\Es$ (cf. section \ref{ss- CAT(0)}). 
Les compactifications obtenues sont alors moins fines.
\end{rema*}

On appellera {\em spectre marqu{\'e} des vecteurs de translation} d'une
re\-pr{\'e}\-sen\-ta\-tion $\rho:\fleche G$ la fonction
$v\circ\rho:\Ga\fleche \Cb$. 
L'espace $\CGa$
est muni de la topologie produit, qui est m{\'e}trisable de type
d{\'e}nombrable (mais pas localement compacte).
Le groupe $\Aut(\Ga)$ agit naturellement 
sur $\CGa$.
L'application $\Vhat:\rho\mapsto v\circ \rho$ de $\R$ dans $\CGa$ est
$\Aut(\Ga)$-{\'e}quivariante, invariante sous l'action par conjugaison de
$G$, et continue (car $v$ l'est, proposition \ref{prop v continue}).
Elle passe donc au quotient en une
application continue $\Aut(\Ga)$-{\'e}quivariante
$$\V:\begin{array}{ccl} 
\Xsep & \fleche & \CGa\\[0ex]%
[\rho] & \mapsto & v\circ\rho
\end{array}\; .$$

\subsection{Le th{\'e}or{\`e}me de compactification}
Soit $\PCGa$ le projectifi{\'e} de $\CGa$, 
c'est-{\`a}-dire l'espace topologique quotient $\CGao$
par la relation d'{\'e}\-qui\-va\-lence $u \simeq w$ 
si et seulement s'il existe $t\in\RR_{>0}$ tel que $u=t w$. 
Cet espace est  s{\'e}par{\'e}, {\`a} base d{\'e}nombrable (mais pas localement compact).
On note $[w]$ la classe dans $\PCGa$ de $w\in\CGao$, et
$\PP:\CGao\fleche\PCGa$ la projection canonique, qui est continue et
ouverte.
L'action de $\Aut(\Ga)$ sur $\CGa$ passe au quotient en une action
(par hom{\'e}omorphismes) sur $\PCGa$.
\begin{theo} 
\label{theo- comp}
\begin{enumerate}
\item Le sous-ensemble $\Xsepo=\V^{-1}(0)$ est un compact.

\item L'application continue
$$\PV=\PP\circ\V :\Xsep - \Xsepo \fleche \PCGa \;.$$ est ``d'image
  relativement compacte {\`a} l'infini'', i.e.  il existe un compact
  $\Ksep$ de $\Xsep$, contenant $\Xsepo$, tel que $\PV(\Xsep -\Ksep)$
  est relativement compact.

En particulier, elle induit 
une compactification (m{\'e}trisable) $\Xsept$
de $\Xsep$, dont le bord $\bordinf\Xsept=\Xsept-\Xsep$
s'identifie {\`a} une partie  de $\PCGa$, et caract{\'e}ris{\'e}e 
par la propri{\'e}t{\'e} suivante :
 une suite $[\rok]_k$ dans $\Xsep$ converge dans $\Xsept$ vers $[w]$
 dans $\bordinf\Xsept\subset\PCGa$ si et seulement si $[\rok]_k$
sort de tout compact de $\Xsep$
et $[\vrok]\tend [w]$ dans $\PCGa$.

\item
L'action naturelle de $\Out(\Ga)$ sur $\Xsep$ se prolonge contin{\^u}ment
au compactifi{\'e} $\Xsept$, avec $[\phi].[w]=[w\circ\phi^{-1}]$ pour tous
$\phi\in\Aut(\Ga)$ et $w\in\CGa$.

\item Si $[w]\in \bordinf\Xsept$, alors 
$[w]=[\vro]$ avec $\rho$ une
action de $\Ga$ sur un immeuble affine $\Delta$ de type $(\Aa,\Wvect)$, sans
point fixe global.
On peut de plus supposer que cette action provient d'une
repr{\'e}sentation non born{\'e}e $\rho :\Ga\fleche \Gb(\KKom)$, avec $\KKom$
un corps valu{\'e} ultram{\'e}trique et $\Delta$ un sous-immeuble 
de l'immeuble de Bruhat-Tits de $\SLn(\KKom)$.  
\end{enumerate}
\end{theo}

 \begin{remas*} 
1. La compactification obtenue ne d{\'e}pend pas du choix du compact
 $\Ksep$. Le corps $\KKom$ est complet mais sa valuation n'est pas a
 priori discr{\`e}te (et l'immeuble impliqu{\'e} n'est pas simplicial).

2. (Un crit{\`e}re de rigidit{\'e})
En particulier, s'il n'existe pas d'action de $\Ga$ sans point fixe
global sur un immeuble affine de type $(\Aa,\Wvect)$, ou s'il n'existe
pas de repr{\'e}sentation $\rho$ non born{\'e}e de $\Ga$ dans $\Gb(\KKom)$
avec $\KKom$ corps valu{\'e} ultram{\'e}trique, alors l'espace $\Xsep$ est compact.
3. (Injectivit{\'e} du spectre)
L'application $\V$ n'est a priori pas injective sur $\Xsep$ tout
entier, ne serait-ce que parce que toutes les repr{\'e}sentations {\`a}
valeurs dans un sous-groupe compact $K$ de $G$ ont le m{\^e}me spectre (nul).
On renvoie {\`a} \cite{DaKi} pour des r{\'e}sultats sur l'injectivit{\'e} (qui
n'est pas utilis{\'e}e ici) du spectre des longueurs $\mathcal{L}$ (donc
de $\V$) et de son projectifi{\'e} $\PP\mathcal{L}$ (donc de
$\PV$) pour les repr{\'e}sentations Zariski-denses dans $G$ semisimple,
dans le cas o{\`u} $\KK=\RR$.
\end{remas*}
Avant de d{\'e}montrer ce th{\'e}or{\`e}me (en section \ref{ss- demo theo comp}),
on commence par expliquer dans un cadre purement topologique la
construction de la compactification associ{\'e}e {\`a} une application d{\'e}finie
hors d'un compact, continue, d'image relativement compacte {\`a} l'infini.

\subsection{Pr{\'e}liminaires topologiques sur la notion de compactification}
\label{ss- methode compactification}

Dans cette section $\E$ d{\'e}signe un espace topologique quelconque.

\begin{defi} Une {\em compactification} de $\E$ est une paire $(g,\Et)$, o{\`u}
$\Et$ est un espace topologique compact et $g : \E \longrightarrow
\Et$ un hom{\'e}omorphisme sur son image, avec $g(\E)$ ouvert et dense
dans $\Et$. Le compl{\'e}mentaire de $g(\E)$ dans $\Et$ est appel{\'e} le {\em
bord ({\`a} l'infini)} de $\Et$ et not{\'e} $\bordinf \Et$.  \end{defi}
\begin{rema*}
Pour que $\E$ admette une compactification
(m{\'e}trisable), il est n{\'e}cessaire qu'il soit localement
compact ({\`a} base d{\'e}nombrable), en particulier s{\'e}par{\'e}.
\end{rema*}

On suppose d{\'e}sormais que $\E$ est localement compact, non compact.
On pose $\Ehat=\E \cup \{ \infty \}$, muni de la topologie {\'e}tendant
celle de $\E$ o{\`u} les compl{\'e}mentaires des compacts de $\E$ forment une
base de voisinages de $\infty$. Alors $(\id, \Ehat)$ est une
compactification de $\E$, appel{\'e}e {\em compactification d'Alexandroff}
de $\E$.
Si $\E$ est m{\'e}trisable, d{\'e}nombrable {\`a} l'infini, ou, de mani{\`e}re
{\'e}quivalente, si $\E$ est {\`a} base d{\'e}nombrable, alors $\Ehat$ est
m{\'e}trisable.
Soit $C$ un espace topologique s{\'e}par{\'e} et $f:\E\fleche C$ continue.
On dira que $\tilde{f} :\Et \fleche C$ continue est un {\em prolongement
de $f$ {\`a} $\Et$} si $\tilde{f}\circ g = f$. 
Le prolongement est unique s'il existe. 
On dit qu'une compactification $(g_1, \E_1)$ de $\E$ est {\em plus
fine} qu'une compactification $(g_2, \E_2)$ de $\E$ si $g_2$ poss{\`e}de
un (unique) prolongement continu $h$ {\`a} $\E_1$ (qu'on appellera le {\em
morphisme canonique}).  On a alors $h(\E_1)= \E_2$ et
$h(\bordinf\E_1)= \bordinf\E_2$.
Les compactifications $\E_1$ et
$\E_2$ sont {\em isomorphes} (i.e. $\E_1$ est plus fine que $\E_2$ et
$\E_2$ est plus fine que $\E_1$) si et seulement si $h$ est un
hom{\'e}omorphisme.
La compactification d'Alexandroff est la compactification la moins
fine.

\paragraph{Compactifier hors d'un compact.}
Soit $\K$ un compact de $\E$ et $\F$ le compl{\'e}mentaire de $\interieur{\K}$ dans
$\E$. Soit $(h,\Ft)$ une compactification de $\F$. Soit $\Et$ l'espace
topologique obtenu en recollant $\K$ et $\Ft$ sur le ferm{\'e} $\bord \K$
au moyen de $h_{|\bord \K}$, 
et $g$ l'application de $\E$ dans
$\Et$, {\'e}gale {\`a} $\id$ sur $\K$ et {\`a} $h$ sur $\F$.
On voit facilement que 
$(g,\Et)$ est une compactification
de $\Et$, dite {\em induite} par $(h,\F)$, dont le bord $\bordinf \Et$
s'identifie  {\`a} celui de $\Ft$.
 Si $(h,\Ft)$ provient (par restriction) d'une compactification
$(g,\Et)$ de $\E$, alors la compactification de $\E$ induite par
$(h,\Ft)$ est isomorphe {\`a} la compactification de d{\'e}part $(g,\Et)$.
\paragraph{Compactification induite par une application continue}
(voir aussi \cite{MoSh}).
Soit $C$ un espace topologique s{\'e}par{\'e} et $f:\E\fleche C$ continue
 d'image relativement compacte.
 Soit $\Ehat=\E\cup\{\infty\}$ le compactifi{\'e} d'Alexandroff de $\E$.
 Soit $g : \E \fleche \Ehat\times C$ l'application continue $x\mapsto
 (x,f(x))$, et $\Et$ l'adh{\'e}rence de $g(\E)$,
qui est
compacte, et de la forme $\Et=g(\E) \cup (\{\infty\}\times B)$ avec
$B\subset C$. Alors $(g,\Et)$ est une compactification de $\E$, dite
{\em associ{\'e}e {\`a}} ou {\em induite par} $f$, 
dont le bord s'identifie {\`a} $B\subset C$. 
Soit $C'$ un autre espace topologique s{\'e}par{\'e} et $h :C\fleche C'$
 continue. Alors $f'=h\circ f$ induit une compactification $(g',\Et')$
de $\E$ moins fine que $\Et$.

Plus g{\'e}n{\'e}ralement, supposons que $f:E\fleche C$ est seulement 
d{\'e}finie sur le
com\-pl{\'e}\-men\-tai\-re $\E-\A$ d'une partie relativement compacte $\A$
de $\E$,
continue, 
et d'image relativement compacte {\`a} l'infini,  
i.e il existe un compact $\K$ de $\E$ tel que $f(\E-\K)$ est relativement
compact dans $C$.
Alors, d'apr{\`e}s ce qui pr{\'e}c{\`e}de, pour tout compact $\K$ de $\E$
contenant $\A$ dans son int{\'e}rieur avec $\ov{f(\E-\interieur{\K})}$ compact,
$f$ d{\'e}finit une compactification
$\Ft$ de $\F=\E-\interieur{\K}$, donc une compactification $\Et$ de $\E$,
dite {\em induite} par $f$. Cette compactification ne d{\'e}pend pas du
choix de $\K$ ({\`a} isomorphisme pr{\`e}s).

Si $\E$ et $C$ sont {\`a} base d{\'e}nombrable, alors $\Et$ est m{\'e}trisable.
Une suite $\xk$ de $\E$ converge dans $\Et$ vers $b$ dans $\bord\Et$
identifi{\'e} {\`a} $B\subset C$ si et seulement si $\xk$ sort de tout compact
de $\E$ et $f(\xk)\tend b$ dans $C$.

\subsection{Fonctions de d{\'e}placement}
\label{ss- deplacement}

On introduit maintenant quelques outils suppl{\'e}\-men\-taires venant de la
g{\'e}om{\'e}trie $\CAT(0)$, dont on aura besoin en section \ref{ss- demo theo
  comp} pour d{\'e}montrer le th{\'e}or{\`e}me de compactification.

Soit $S$ une partie g{\'e}n{\'e}ratrice finie de $\Ga$, et $\lmotsS{\cdot}$ la longueur des mots sur $\Ga$ associ{\'e}e. 
On voit $\rho:\Ga\longrightarrow G$ comme 
une action de $\Ga$ sur $\Es$ (l'espace $\CAT(0)$  associ{\'e} {\`a} $G$).
On appelle {\em fonction de d{\'e}placement} de $\rho$
(relativement {\`a} la partie g{\'e}\-n{\'e}\-ra\-tri\-ce $S$) et on note $\dro$ la
fonction convexe continue 
$$\begin{array}{cccl}\dro :& \Es & \longrightarrow &\RR_+ \\
                     & x & \mapsto & \sqrt{\sumS d(x,\rho(s)x)^2}
\end{array}$$ 
On appelle {\em minimum de d{\'e}placement} de $\rho$ (relativement {\`a} la
partie g{\'e}\-n{\'e}\-ra\-tri\-ce $S$) et on note $\laro$ le nombre r{\'e}el positif
$$\laro=\inf_{x\in \Es}\dro(x),$$ qui est un invariant de conjugaison.
 La fonction $\la : \R\fleche \RRp$ va nous permettre de contr{\^o}ler les
  repr{\'e}sentations partant {\`a} l'infini dans le quotient $\Xsep$,
 et de renormaliser le spectre des vecteurs $\V$.

\begin{prop} \label{prop- majoration de la lT}
Pour tout $\rho\in\R$,
on a 
\(\ \ell\circ\rho\leq \laro\lmotsS{\cdot}\ \)
 sur $\Ga$.
\qed\end{prop}

Pour tout r{\'e}el positif $D$, notons 
$\RleqD=\{\rho\in\R\ |\ \dro(\xo)\leq D\}$,
qui est un compact de $\R$ (car
l'action de $G$ sur $\Es$ est propre), et $\XsepD=\psep(\RleqD)$, qui
est un compact de $\Xsep$. 
On d{\'e}montre facilement la propri{\'e}t{\'e} suivante.
 
\begin{prop}
  \label{prop- lambda borne implique compact}
Il existe $L>0$ tel que, pour tout $D>0$, si $\laro < D$, alors 
il existe $g\in G$ tel que $g\cdot \rho \in \Rleq {D+L}$. 
En particulier $[\rho]$ est dans le compact $\Xsepleq {D+L}$.
\qed\end{prop}
\begin{rema*}
On peut en fait montrer  que $\la$
est continue sur $\R$ et passe au quotient en une fonction continue
propre sur $\Xsep$, voir \cite[Proposition 25]{ParRepcr}.
 \end{rema*}

\subsection{La d{\'e}monstration du th{\'e}or{\`e}me  de compactification \ref{theo- comp}}
\label{ss- demo theo comp}

On commence par {\'e}tablir une version un peu plus faible, 
qui d{\'e}\-crit le comportement asymptotique  
d'une suite de repr{\'e}sentations partant {\`a} l'infini  dans $\Xsep$.

\begin{theo}
\label{thm- larok tend vers l'infini}
 Soit $(\rok)_{k\in\NN}$
une suite dans $\R$ telle que $\lak=\larok\tend \pinfty$.
Alors, quitte {\`a} extraire, $\lslak \vrok\tend w$ dans $\CGa$, 
avec $w=\vro$ o{\`u} $\rho$ est une action de $\Ga$ sur un immeuble
affine $\Delta$ de type $(\Aa,\Wvect)$, sans point fixe global.
On peut de plus supposer que cette action provient d'une
repr{\'e}sentation $\rho$ non born{\'e}e de $\Ga$ dans
$\Gb(\KKom)\subset\SLn(\KKom)$, o{\`u} $\KKom$ est un corps valu{\'e}
ultram{\'e}trique
et que $\Delta$ est
un sous-immeuble de l'immeuble de Bruhat-Tits de
$\SLn(\KKom)$, pr{\'e}serv{\'e} par $\Gb(\KKom)$.

En particulier, $w$ n'est pas identiquement nul.  
\end{theo}

\begin{preuve} 
Comme $G$ agit cocompactement sur $\Es$, quitte {\`a} conjuguer chaque
repr{\'e}sentation $\rok$, on peut supposer qu'il existe un point $\xk$ de $\Es$,
{\`a} distance born{\'e}e par $M>0$ de $\xo$, tel que $\drok (\xk) \leq \larok
+ {1 \over i}$.
Soit $\om$ un ultrafiltre sur $\NN$, plus fin que le filtre de
Fr{\'e}chet.  On consid{\`e}re le c{\^o}ne asymptotique $\Esom$ de $(\Es, \xo,\ {1
\over \lak}d)$ suivant l'ultrafiltre $\om$ (voir section \ref{ss- prel
conas}). C'est un immeuble affine complet de type $(\Aa,\Wvect)$ (cf. Thm
\ref{theo- cone as = immeuble}).
Pour tout $\ga$ de $\Ga$, on a $d(\xo,\ \rok(\ga)\xo) \leq
\lmotsS{\ga}\drok(\xo) \leq \lmotsS{\ga} (\larok + {1 \over i}+2M\sqrt{|S|})$.
donc ${1 \over \lak}d(\xo,\rok(\ga)\xo)$ est born{\'e} pour $k\in \NN$.
On peut donc consid{\'e}rer l'action asymptotique $\room$ de $\Ga$ sur
$\Esom$, qui est d{\'e}finie par
 $\room(\ga) [\xk]=[\rok(\ga) \xk]$ pour tous $\ga$ de $\Ga$ et
$[\xk]$ de $\Esom$.
Pour tout $\yom=[\yk]\in\Esom$, on a
$\droom(\yom)=\limom\lslak\drok(\yk)$,
donc $\droom(\yom)\geq 1$, avec {\'e}galit{\'e} si $\yom=[\xk]=[\xo]$. 
Donc $\room$ n'a pas de point fixe global dans $\Esom$.
Comme $v$ passe contin{\^u}ment au c{\^o}ne asymptotique (proposition
\ref{prop- passage au cone}), pour tout $\ga$ de $\Ga$, on a $\lslak
v(\rok(\ga))\tendom v(\room(\ga))$ dans $\Cb$.
En particulier 
il existe  une sous-suite extraite de $\lslak \vrok$ convergeant vers
$\vroom$ dans $\CGa$.  
Dans le cas o{\`u} $\KK=\RR$ ou $\CC$, on peut 
supposer
que $G$ est un sous-groupe ferm{\'e} auto-adjoint de $\SLn(\KK)$. 
En particulier $\Es$ est un sous-espace totalement g{\'e}od{\'e}sique passant
par $\xo$ de l'espace sym{\'e}trique $\Esn$ associ{\'e} {\`a} $\SLn(\KK)$.
Dans le cas o{\`u} $\KK$ est ultram{\'e}trique, on peut supposer que $\Gb$ est
un sous-groupe alg{\'e}brique d{\'e}fini sur $\KK$ de $\SLn$, et d'apr{\`e}s le
th{\'e}or{\`e}me de plongement de Landvogt \cite[Thm. 2.2.1]{Landvogt},
l'immeuble de Bruhat-Tits $\Es$ de $\Gb$ sur $\KK$ se plonge par une
isom{\'e}trie {\'e}quivariante par $G$ dans l'immeuble de Bruhat-Tits $\Esn$
de $\SLn$ sur $\KK$. On identifie $\Es$ {\`a} son image dans $\Esn$.

Le c{\^o}ne asymptotique $\Esom$ de $(\Es,\xo, \lslak d)$ suivant l'ultrafiltre
$\om$ est donc un sous-immeuble du c{\^o}ne asymptotique $\Esnom$ de $(\Esn,\xo,
\lslak d)$ suivant l'ultrafiltre $\om$.
D'apr{\`e}s ce qui a {\'e}t{\'e} fait dans la section \ref{s- modele alg de Esom}
(voir le corollaire \ref{coro- isom SLn(KKom)} et la remarque
qui suit, et le th{\'e}or{\`e}me \ref{theo- modele alg Esom}), l'action
asymptotique provient en fait
d'une repr{\'e}sentation
$\room$ de $\Ga$ dans $\Gb(\KKom)\subset\SLn(\KKom)$, o{\`u} $\KKom$ est un corps
valu{\'e}  ultram{\'e}trique
et $\SLn(\KKom)$ agit sur son immeuble de
Bruhat-Tits.
Comme cette action n'a pas de point fixe global dans $\Esom$, donc dans
$\Esnom$, 
qui est un immeuble affine complet (cf. Thm \ref{theo- cone as =
immeuble}), $\room(\Ga)$ n'est pas born{\'e}, et le spectre des longueurs
 $\ell\circ\room$ ne peut {\^e}tre nul \cite[Coro. 3]{ParEll}, donc
$w=\vroom$ n'est pas  nul.
\end{preuve} %
On en d{\'e}duit imm{\'e}diatement la propri{\'e}t{\'e} suivante, qui permet de con\-tr{\^o}\-ler
le spectre des longueurs $\mathcal{L}$ (donc {\'e}galement celui des vecteurs $\V$)
en fonction de $\la$,  {\`a} l'infini de $\Xsep$.

\begin{coro}
\label{coro- controle fin de V hors de XsepD}
Il existe une partie finie $F$ de $\Ga$, un r{\'e}el $c>0$ et un r{\'e}el
$D_0\geq 0$ tels qu'on a la propri{\'e}t{\'e} suivante.

Pour tout $\rho\in\R$, si $\laro\geq D_0$, il existe $\ga\in F$ tel que
$\ell(\rho(\ga)) > c\laro$. 
\end{coro}

\begin{preuve}
Soit $F_k$, $k\in\NN^*$ une suite croissante de parties finies
recouvrant $\Ga$. Supposons par l'absurde que pour tout $k\in\NN^*$ non nul, 
il existe $\rok\in\R$ tel que
$\la(\rok)\geq k$ et pour tout
$\ga\in F_k$ on a $\lsla\ell(\rok(\ga))\geq \frac{1}{k}$. 
Alors
$\lsla\ell\circ\rok \tend 0$ dans $\RRp^\Ga$, 
ce qui est impossible par le th{\'e}or{\`e}me \ref{thm- larok tend vers l'infini}.
\end{preuve}

\begin{remas*}
1)  On a donc (pour $\laro\geq D_0$) l'encadrement
$$c\laro<\max_{\ga\in F}\ell(\rho(\ga)) \leq c'\laro$$
o{\`u} $c'=\max_{\ga\in F}\lmotsS{\ga}$ (cf proposition \ref{prop-
    majoration de la lT}).

2) Une autre cons{\'e}quence directe du corollaire
\ref{coro- controle fin de V hors de XsepD}, que nous n'utiliserons pas, mais qui
m{\'e}rite d'{\^e}tre signal{\'e}e, est que
l'application $\V:\Xsep \fleche \CGa$ est propre~ :
en effet, si $Q$ est un compact de $\CGa$, pour $D$ sup{\'e}rieur {\`a} $D_0$
et {\`a} $\frac{1}{c}\max_{\ga\in F,  w\in Q}\norme{w(\ga)}$,
on a que $\la < D$ sur $\V^{-1}(Q)$. 
Donc $\V^{-1}(Q)$ est inclus dans le compact $\Xsepleq {D+L}$
(proposition \ref{prop- lambda borne implique compact}).

\end{remas*}

\begin{preuve}[Fin de la preuve du th{\'e}or{\`e}me \ref{theo- comp}]
  Nous pouvons maintenant montrer l'ex\-is\-ten\-ce du compact $\Ksep$
  tel que $\Xsepo\subset \K$ et $\PV(\Xsep-\Ksep)$ est inclus dans un
  compact.  
Soit $D_0$ donn{\'e} par le corollaire \ref{coro- controle fin de V hors
  de XsepD} et $L$ donn{\'e} par la propostion \ref{prop- lambda borne
  implique compact}. 
On prend alors $\Ksep=\XsepD$ pour $D\geq D_0+L$.
Notons $\Y=\Xsep-\Ksep$ et $\Yhat=\psep^{-1}(\Y)=\R -\psep^{-1}(\Ksep)$.
D'apr{\`e}s le corollaire \ref{coro- controle fin de V hors de XsepD},
comme $\rho\in\Yhat$ implique $\laro \geq D-L \geq D_0$ (d'apr{\`e}s la
proposition \ref{prop- lambda borne implique compact}), on a que
$\lsla\Vhat(\Yhat)$ ne rencontre pas l'ensemble
$\{w\in\CGa\ |\ \forall \ga\in F,\ \norme{w(\ga))} \leq c\}$
qui est un voisinage de $0$ dans $\CGa$ (en particulier $\Ksep$ contient $\V^{-1}(0)$).
Donc l'ad\-h{\'e}\-ren\-ce  $C$  de 
$\lsla\Vhat(\Yhat)$ dans $\CGa$ ne contient pas $0$.
Or $\lsla\Vhat$ est d'image relativement compacte,
car incluse dans le compact 
$\prod_{\ga\in\Ga} 
\{u\in \Cb;\    \norme{u} \leq \lmotsS{\ga}\}$ 
par la proposition \ref{prop- majoration de la lT}.
Donc $C$ est un compact de $\CGao$, et $\PP(C)$ est un compact de
$\PCGa$, contenant
$\PP(\lsla\Vhat(\Yhat))=\PP(\Vhat(\Yhat))
=\PP(\V(\Y))=\PV(\Xsep-\Ksep)$, ce qui
conclut.

La construction (purement topologique) de la compactification de
$\Xsep$ induite par $\PV$ 
a {\'e}t{\'e} d{\'e}crite en d{\'e}tail en section \ref{ss- methode
  compactification}.
Le reste des assertions du th{\'e}or{\`e}me \ref{theo- comp} d{\'e}coule du th{\'e}or{\`e}me
\ref{thm- larok tend vers l'infini}.

Pour  l'action de $\Out(\Ga)$ : comme  l'application 
$\PV:\Xsep - \V^{-1}(0) \fleche \PCGa$  est $\Aut(\Ga)$-{\'e}quivariante, 
 l'action de $\Aut (\Ga)$ sur $\Xsep$  se prolonge continument {\`a}
$\Xsept$ par l'action naturelle sur $\PCGa$.
Comme $\Int(\Gamma)$
agit trivialement sur $\Xsep$, donc sur $\Xsept$, l'action induit une action de $\Out(\Ga)$ sur $\Xsept$ par hom{\'e}omorphismes.
\end{preuve}

\section{Repr{\'e}sentations fid{\`e}les et discr{\`e}tes}
\label{ss- Rfd -ES}

 Soit $\Ga$ un groupe infini de type fini, discret, et $G$ un groupe
 topologique localement compact.
 On dit qu'une repr{\'e}sentation de $\Ga$ dans $G$ est {\em fid{\`e}le et
discr{\`e}te} (fd) si elle injective et d'image discr{\`e}te. 
On note $\RfdGaG$ ou $\Rfd$ l'espace des re\-pr{\'e}\-sen\-ta\-tions
fd de $\Ga$ dans $G$.  

\subsection{L'hypoth{\`e}se (H)}
On dira qu'un groupe abstrait $\Gamma$ {\em v{\'e}rifie l'hypoth{\`e}se \textup{(H)}}
s'il n'admet pas de sous-groupe d'indice fini contenant un sous-groupe
ab{\'e}lien distingu{\'e} infini. 
Cette hypoth{\`e}se n'est pas tr{\`e}s restrictive dans notre cadre, comme le
montrent les exemples suivants.

Si $\Gamma'$ est un groupe commensurable {\`a} $\Gamma$ (c'est-{\`a}-dire si
$\Gamma \cap \Gamma'$ est d'indice fini dans $\Gamma$ et dans
$\Gamma'$), alors $\Gamma$ v{\'e}rifie (H) si et seulement si $\Gamma'$ v{\'e}rifie
(H).
Si $\Gamma$ est un sous-groupe discret Zariski-dense (par exemple un
r{\'e}seau) d'un groupe de Lie lin{\'e}aire semimple $H$, alors $\Gamma$
v{\'e}rifie (H) par \cite[Lemma 1.2]{GoMi}.
Plus g{\'e}\-n{\'e}\-ra\-le\-ment, ce qui suit montre que 
les groupes agissant
proprement, par automorphismes, sans orbite finie au
bord, sur un espace sym{\'e}trique sans facteur compact
et sans facteur euclidien 
 v{\'e}rifient l'hypoth{\`e}se (H). 
 C'est par exemple le cas pour les
  sous-groupes fortement irr{\'e}ductibles de $\SLn(\RR)$ (i.e. tels que
  tous les sous-groupes d'indice fini sont irr{\'e}ductibles).
Les groupes divisant des convexes fournissent beaucoup
d'exemples de tels sous-groupes (voir \cite{Benoist2}).

\begin{prop}
On suppose que $\Es$ est un espace sym{\'e}trique sans facteur compact.
Soit $\Ga$ un groupe agissant proprement par automorphismes sur $\Es$, 
poss{\`e}dant un sous groupe ab{\'e}lien $A$ distingu{\'e} infini.  
Alors il y a deux possibilit{\'e}s.
\begin{enumerate}
\item $A$ poss{\`e}de un {\'e}l{\'e}ment parabolique non trivial. Alors  $\Ga$ admet
  un point fixe global dans $\bordinf\Es$.

\item $A$ n'a que des {\'e}l{\'e}ments semisimples. Alors 
$M=\cap_{a\in A} \Min(a)$ est non vide, et  se d{\'e}compose
de mani{\`e}re canonique en un produit $C\times \RR^k$ (avec $k> 0$
minimal), de telle mani{\`e}re que tous les {\'e}l{\'e}ments $a$ de $A$ op{\`e}rent
sur $M$ comme $(\id,a')$, o{\`u} $a'$ est une translation de $\RR^k$.  
Le groupe $\Ga$ stabilise $M$ et pr{\'e}serve $\bordinf\RR^k$.

\end{enumerate}
En particulier $\Ga$ a une orbite finie dans $\bordinf\Es$.
\end{prop}

\begin{preuve}[Preuve de la proposition]
Le premier point
d{\'e}coule de \cite[Corollary 4.4.5]{Ebe}.
Supposons maintenant que $A$ n'a que des {\'e}l{\'e}ments semisimples.
La premi{\`e}re assertion du deuxi{\`e}me point d{\'e}coule de \cite[Lemma 7.1]{BGS}.
Comme $A$ contient au moins un {\'e}l{\'e}ment non elliptique (car sinon $A$
 fixe $M$ non vide donc est fini),
alors $k\neq 0$.
Comme $A$ est distingu{\'e} dans $\Ga$, le groupe $\Ga$ stabilise $M$ et
pr{\'e}serve la d{\'e}composition $M= C\times\RR^k$, donc pr{\'e}serve 
$\bordinf\RR^k$ non vide.
Or $\bordinf\RR^k$ est inclus dans un appartement 
de $\bordinf\Es$.
Le groupe $\Ga$ a donc une orbite finie dans $\bordinf\Es$.  
\end{preuve}

\subsection{Fermeture}
On suppose ici que $\Ga$ v{\'e}rifie l'hypoth{\`e}se $(H)$ et que $G$
est un groupe de Lie r{\'e}el lin{\'e}aire.

\begin{theo}[Goldman-Millson]  
\label{theo- Rfd ferme - espace sym}
L'espace $\Rfd$ des re\-pr{\'e}\-sen\-ta\-tions fi\-d{\`e}\-les et dis\-cr{\`e}\-tes de
$\Ga$ dans $G$ est un ferm{\'e} de l'espace $\R$ des repr{\'e}sentations de
$\Ga$ dans $G$.\end{theo}

\begin{preuve} Si $\Rfd$ est non vide, alors $\Ga$ est lin{\'e}aire donc il
contient un sous-groupe sans torsion $\Ga'$ d'indice fini, d'apr{\`e}s le
lemme de Selberg \cite{Alperin}.  Alors $\Ga'$ n'a pas de sous-groupe
ab{\'e}lien distingu{\'e} non trivial (car il v{\'e}rifie (H) et n'a pas de
sous-groupe fini non trivial), donc, par Goldman-Millson \cite[Lemma
  1.1]{GoMi}, 
l'espace $\Rfd(\Ga',G)$ est ferm{\'e}.

Soit $\rok\in\Rfd$, $k\in\NN$ avec $\rok\tend \rho$ dans $\R$.
La restriction de $\rho$ {\`a} $\Ga'$ est donc fd.
Comme $\ker \rho \cap \Ga' =\{e\}$, on a que $\ker \rho$ est fini.
Or $1\in G$ n'est pas limite d'{\'e}l{\'e}ments d'ordre inf{\'e}rieur {\`a} une
constante donn{\'e}e.
La restriction de $\rok$ {\`a} $\ker \rho$ tend vers la repr{\'e}sentation
triviale, donc est constante {\`a} partir d'un certain rang.  Comme $\rok$
est fid{\`e}le pour tout $k$, on en d{\'e}duit que $\ker \rho$ est trivial,
autrement dit que $\rho$ est fid{\`e}le.
De plus, on voit facilement qu'un sous-groupe de $G$ contenant un
sous-groupe discret d'indice fini est discret, donc $\rho$ est
discr{\`e}te.
\end{preuve}

\subsection{Compactification}
\label{ss- comp Xfd}

 On suppose ici que  $\Ga$ v{\'e}rifie l'hypoth{\`e}se (H), 
 et que $G$ est un groupe r{\'e}ductif r{\'e}el 
(voir section \ref{ss- espace  associe a un GR})
agissant sur son espace sym{\'e}trique associ{\'e} $\Es$.

Comme $\Rfd$ est un ferm{\'e} de $\R$ (Th{\'e}or{\`e}me \ref{theo- Rfd ferme -
  espace sym}), 
l'espace $\Xsepfd=\psep(\Rfd)=(\Rfd\cap\Rcr)/G$ est 
un ferm{\'e} de $\Xsep=\RcrsG$. 
La compactification $\Xsept$ de $\Xsep$ construite pr{\'e}c{\'e}demment
(Th{\'e}or{\`e}me \ref{theo- comp}) induit donc une compactification
(m{\'e}trisable) $\Xsepfdt$ de $\Xsepfd$ (qui s'obtient en prenant
simplement l'adh{\'e}rence de $\Xsepfd$ dans $\Xsept$). Elle est munie
d'une action naturelle de $\Out(\Ga)$ (car $\Xsepfd$ est stable sous
$\Out(\Ga)$, donc son adh{\'e}rence aussi).

En rang $1$, cela redonne (par construction) les compactifications
connues, voir \cite{FLP}, \cite{MoSh}, \cite{Bestvina}, \cite{PauInv}.

Si $x$ est dans $\bordinf \Xsepfdt
\subset\PCGa$, alors
$x=[\vro]$, o{\`u} $\rho$ est une action sans point fixe global de $\Ga$
sur un immeuble affine de type $(\Aa,\Wvect)$, sans point fixe global.
Dans ce cas, F.~Paulin a d{\'e}montr{\'e} que cette action est {\`a}
stabilisateurs de germes d'appartements virtuellement r{\'e}solubles \cite{PauDg}.

\begin{rema*}
En fait, 
on peut construire la compactification de $\Xsepfd$ 
un peu plus directement, car  $\PV$ 
est  d{\'e}finie sur $\Xsepfd$ tout entier, et {\`a} valeurs dans un
compact de $\PCGa$.
En effet, on peut montrer que si $\rho\in\Rfd$ alors
$\ell\circ\rho\neq 0$
(on peut se ramener par
semisimplification {\`a} $\rho$ cr, et cela d{\'e}coule alors du th{\'e}or{\`e}me de
Burnside, voir par exemple \cite{Bass}, \cite{Benoist}).
On peut aussi montrer que la fonction $\la$
reste sup{\'e}rieure {\`a} $\eps>0$ sur $\Rfd$.
\label{rem- lambda non nul sur Rfd}
On montre alors de la m{\^e}me mani{\`e}re qu'en section \ref{ss- demo theo
  comp} que $\lsla\Vhat(\Rfd)$ est d'adh{\'e}rence compacte, incluse
dans $\CGao$.
\end{rema*}

\noindent {\footnotesize  
{\bf Remerciements.} Je remercie Fr{\'e}d{\'e}ric Paulin pour son soutien, 
sa disponibilit{\'e} constante et ses nombreuses suggestions et corrections.
 }

\end{document}